\documentclass[12pt,reqno]{amsart}
\usepackage{amsfonts,amstext}
\usepackage{amsmath,amstext,amssymb,amscd,mathrsfs,amsthm}
\usepackage{enumerate}
\usepackage[dvipsnames,usenames]{xcolor}

\newtheorem{thm}{Theorem}[section]
\newtheorem{lem}[thm]{Lemma}
\newtheorem{cor}[thm]{Corollary}
\newtheorem{prop}[thm]{Proposition}
\newtheorem{assumption}[thm]{Assumption}
\theoremstyle{definition}
\newtheorem{defn}[thm]{Definition}

\theoremstyle{remark}
\newtheorem{rmk}[thm]{Remark}
\numberwithin{equation}{section}

\usepackage[allbordercolors={1 1 1}]{hyperref}
\usepackage[disable]{todonotes}

\def\grad{\nabla}
\newcommand{\W}[1][1,p]{W^{#1}(\Omega)}
\renewcommand{\L}[1][p]{L^{#1}(\Omega)}
\renewcommand{\H}[1][1]{W^{#1,2}(\Omega)}
\newcommand{\Lb}[1][2]{L^{#1}(\Gamma)}
\newcommand{\Wo}[1][1,p]{W_0^{#1}(\Omega)}
\newcommand{\into}{\hookrightarrow}
\newcommand{\tinto}{\mathrel{\overset{\gamma}{\to}}}
\newcommand{\R}{\mathbb{R}}
\newcommand{\N}{\mathbb{N}}
\newcommand{\E}{\mathscr{E}}
\newcommand{\D}{\mathscr{D}}
\newcommand{\A}{\mathscr{A}}
\let\div\undefined
\DeclareMathOperator{\div}{div}
\DeclareMathOperator{\spn}{span}
\newcommand{\ptag}[1]{ \tag{\ref{#1}$'$} } 
\DeclareMathOperator*{\esssup}{ess\,sup}
\def\t{\gamma} %

\def\g{\gamma}

\def\G{\Gamma}
\def\O{\Omega}

\def\grad{\nabla}

\newcommand{\plap}{\Delta_p}

\author[N. J. Kass]{Nicholas J. Kass}
\address{Department of Mathematics, University of Nebraska--Lincoln, Lincoln, NE  68588-0130, USA} \email{nkass@huskers.unl.edu}

\author[M. A. Rammaha]{Mohammad A. Rammaha}
\address{Department of Mathematics, University of Nebraska--Lincoln, Lincoln, NE  68588-0130, USA} \email{mrammaha1@unl.edu}

\thanks{This research was partially supported by NSF grant DMS-1211232.}

\title[Wave equations of the $p$-Laplacian type]{On Wave equations of the $p$-Laplacian type with supercritical nonlinearities}
\date{\today}
\subjclass[2010]{Primary: 35L05, 35L20, 35L72 Secondary: 58J45}
\keywords{wave equation, $p$-Laplacian, supercritical sources, local existence, generalized Robin condition}

\begin{document}

\begin{abstract} This article focuses  on  a quasilinear wave equation of $p$-Laplacian type:
	\[
	u_{tt} - \Delta_p u -\Delta u_t = f(u)
	\]
	in a bounded domain $\O \subset \R^3$ with a sufficiently smooth boundary $\G=\partial \O$ subject to a generalized Robin boundary
	condition featuring boundary damping and a nonlinear source term. The operator $\plap$, $2<p<3$, denotes the classical $p$-Laplacian. The interior and boundary terms $f(u)$, $h(u)$  are sources  that  are allowed to  have a \emph{supercritical} exponent, in the sense that their associated Nemytskii operators are not locally Lipschitz from $\W$ into $L^2(\O)$ or $L^2(\G)$.  
Under suitable assumptions on the parameters we provide a rigorous  proof of existence of a local weak solution which can be extended globally in time, provided the damping terms dominates the corresponding sources in an appropriate sense. Moreover, a blow-up result is proved for solutions with negative initial total energy.
\end{abstract}

\maketitle

\section{Introduction}\label{S1}

\subsection{The model} 
This paper is concerned with the  existence of local and global solutions to the quasilinear initial-boundary value problem:
\begin{align}
\label{wave}
\begin{cases}
u_{tt}-\Delta_p u -\Delta u_t = f(u) &\text{ in } \Omega \times (0,T),\\[.1in]
(u(0),u_t(0))=(u_0,u_1) \\[.1in]
|\grad u|^{p-2}\partial_\nu u + |u|^{p-2}u + \partial_\nu u_t + u_t = h(u)&\text{ on }\Gamma \times(0,T),
\end{cases}
\end{align}
for given initial data $(u_0,u_1) \in \W\times\L[2]$ and $2<p<3$.  
The operator $\Delta_p$ is the classical $p$-Laplacian given by: 
\begin{align*}
\Delta_p u=\div(|\grad u|^{p-2}\grad u).
\end{align*}
For the sake of physical relevance we shall assume that $\O\subset\R^3$ is a bounded open domain with boundary $\G$ of class $C^2$ having outward normal vector $\nu$, however analogous results are possible in other spatial dimensions provided the appropriate changes are made to the exponents coming from the various Sobolev embedding and trace theorems.  

Throughout the paper, we study (\ref{wave}) under the following  assumption:
\begin{assumption}\label{ass:fg}
We assume that the interior and boundary source feedback terms $f,\, h\in C^1(\R)$ are $\R$-valued functions such that 
\begin{align*}
		|f'(u)|&\leq C(|u|^{q-1}+1)\text{   where   } 1\leq q < \frac{5p}{2(3-p)}, \\
		|h'(u)|&\leq C(|u|^{r-1}+1)\text{   where   } 1\leq r < \frac{3p}{2(3-p)},
\end{align*}
where $2<p<3$.
\end{assumption}

\begin{rmk}\label{rmk:fbound}
These restrictions on the exponents $q$ and $r$ are inherited from the problem itself and the  Sobolev embedding and trace theorems.  In addition, 
	as the bounds will be used often throughout the paper it is worthy of note that the assumptions on $f$ and $h$ imply that
\begin{alignat*}{4}
	|f(u)|&\leq C(|u|^q+1),\quad&|f(u)-f(v)|&\leq C(|u|^{q-1}+|v|^{q-1}+1)|u-v|,\\
	|h(u)|&\leq C(|u|^r +1),\quad&|h(u)-h(v)|&\leq C(|u|^{r-1}+|v|^{r-1}+1)|u-v| .		
	\end{alignat*}	
\end{rmk}

\subsection{Literature overview and new contributions}
Strongly damped wave equations of the form
\begin{align}\label{lit-wave}
u_{tt} -\Delta u - \Delta u_t = f
\end{align}
have been given significant attention in the literature, in part due to their natural physical interpretations as modeling vibrations in viscoelastic materials.  In fact, the term $-\Delta u_t$ in \eqref{lit-wave} is commonly referred to as Voigt damping in reference to its role in describing so-called Kelvin-Voigt materials, exemplified in one dimension as a viscous damper in parallel with an elastic spring.  The source feedback term, $f(u)$, is permitted to have ``bad" sign, in that its presence may serve to increase the total energy of the system in time. In general, it is the relative strength of this source term as compared to the damping which will determine the long-term behavior of the equation, and thus the interaction between the two which is of particular interest.

Beginning in 1980, a seminal paper by Webb, \cite{W}, establishes unique global solutions exhibiting exponential decay of energy to an equation of the form \eqref{lit-wave}  with zero Dirichlet boundary condition for an essentially linear function $f$ using the theory of semigroups. A significant generalization is given in 1991 by Ghidaglia and Marzocchi in \cite{MR1112054}, where the Laplacian is replaced by a positive, linear operator $A$ and the requirements on the source feedback term $f$ are greatly relaxed, permitting sources of order five in the form $f(u)=C(1+u^5)$ in dimension three.  With these less restrictive conditions on the source feedback term 
solutions need no longer be global.

Generalizing this work to equations in which the principal part of the PDE contains a nonlinearity such as the $p$-Laplacian introduces additional challenges.  In \cite{MR1634008}, Chen, Guo, and Wang investigate the end behavior of solutions of 
\begin{align}\label{lit-chenguo}
u_{tt} -\sigma(u_x)_x - u_{xxt} =f(u)+ g(x)
\end{align}
with zero Dirichlet boundary where $\sigma$ is a smooth function satisfying $\sigma(0)=0$  along with the bound $\sigma'(s)\geq r_0>0$ for all $s\in\R$.  By taking $\sigma(s)=s$ in \eqref{lit-chenguo} the problem under consideration is a damped wave equation in dimension one whose principle part matches \eqref{lit-wave}.  While this formulation does permit some nonlinearity it does not include the case of the $p$-Laplacian as the function $\sigma(s)=|s|^{p-2}s$ and its derivative do not enjoy the necessary smoothness or boundedness.

Biazutti's work in \cite{Bia:95:NA} involves a Cauchy problem of the form 
\begin{align}
u_{tt}(t)+Au(t) + Gu_t(t)  + B(t)u_t(t)= f(t)
\end{align}
for a linear operator $B(t)$ and nonlinear operators $A$ and $G$.  

The work of Rammaha and Wilstein in \cite{RW} and Pei et. al. in \cite{PRT-p-Laplacain} explicitly include the $p$-Laplacian by considering, respectively, equations of the form 
$$u_{tt}-\Delta u - \Delta_p u_t =f(u)\quad\text{and}\quad u_{tt}-\Delta_p u - \Delta u_t =f(u)$$
with zero Dirichlet boundary in dimension three with $2<p<3$.  In both works the assumptions on the source $f$ are quite mild, and $f$ is permitted to have so-called supercritical order in that it is no longer locally Lipschitz continuous when viewed as a map from the solution space, $\Wo[1,p]$, into $\L[2]$.  

Works which include boundary conditions are not as well represented in the literature. A closely related problem is studied  by Vitillaro \cite{V3,V2}:
\begin{align*}
\begin{cases}
u_{tt}-\Delta u=0&\text{ in }\O\times(0,T),\\
u=0&\text{ on }\G_0\times [0,T),\\ 
\partial_\nu u + |u_t|^{m-2}u_t = |u|^{p-2}u  &\text{ on }\G_1\times (0,T) 
\end{cases}
\end{align*}
where $\O$ is a domain in $\R^n$ with smooth boundary given as the disjoint union $\partial\O = \G_1\cup \G_2$. 

In addition to the references given in detail above, the literature is rich with results on wave equations and systems of wave equations. Many pioneering papers such as Lions and Strauss \cite{LSt}, as well as works by Glassey \cite{G} and Levine \cite{l1} are worthy of mention.  More recently, Georgiev and Todorova \cite{GT} ignited significant interest in the interaction of source and damping terms in wave equations. The blow-up result in \cite{BLR3} further characterizes this  interaction in the case where the damping is nonlinear.  For systems of wave equations in bounded domains the papers \cite{AR2,GR1,GR2,GR} are additionally worthy of mention.  

In this manuscript we employ a Galerkin type scheme to demonstrate the existence of suitably defined weak solutions to \eqref{wave}, and then prove sufficient conditions for global stability as well as a blow-up result in finite time.

 Several technical challenges are present, chiefly involving the identification of the limiting value of $\Delta_p u_N$ with the value of $\Delta_p u$ at the Galerkin level which we carefully accomplish through the use of monotone operator theory. Our detailed approach also highlights the crucial difficulty that would arise if the Kelvin-Voigt damping were replaced with an $m$-Laplacian term $\Delta_m u_t$, $m>2$. In that case, the simultaneous identification of two weak limits, one for the $p$-Laplacian of $u$ and the other for the $m$-Laplacian of $u_t$ (even if $m=p$) cannot be carried out by the same approach.  It had been assumed in some previous works that the Galerkin approach might trivially extend to the $m$-$p$ model, for instance in \cite{BM2} which attempts to rely on \cite{Bia:95:NA}  and \cite{nak-nan:75} that deal with a \emph{single} $p$-Laplace operator in the equation.
 That is not the case, however, and rigorous analysis of well-posedness for $p$-Laplacian/$m$-Laplacian (with $m,\, p >2$) second-order equation is presently missing from the literature, remaining a challenging open problem.

\subsection{Notation}
Throughout the paper the following notational conventions for $L^p$ space norms and inner products will be used, respectively:
\begin{align*}
&||u||_s=||u||_{\L[s]}, &&|u|_s=||u||_{\Lb[s]};\\
&(u,v)_\Omega = (u,v)_{\L[2]},&&(u,v)_\Gamma = (u,v)_{\Lb[2]}.
\end{align*}
We also use the notation  $\g u$ to denote the \emph{trace} of $u$ on $\G$ and we write $\frac{d}{dt}(\g u(t))$ as $\g u_t$ or $\g u'$.
\\ As is customary, $C$ shall always denote a positive constant which may change from line to line.  Following from the Poincar\'{e}-Wirtinger type inequality
\[
||u||_p^p \leq C(||\grad u||_p^p +|\g u|_p^p)\text{ for all }u\in \W
\]
we may choose as a matter of convenience
\[
||u||_{1,p}=\left(||\grad u||_p^p +|\g u|_p^p\right)^{1/p}
\]
as a norm on $\W$ equivalent to the standard norm.\\
For a Banach space $X$, we denote the duality pairing between the dual space $X'$ and $X$ by $\langle \cdot,\cdot \rangle_{X',X}$. That is, 
\begin{align*}
\langle \psi,x \rangle_{X',X} =\psi(x)\text{ for }x\in X,\, \psi\in X'.
\end{align*}
In particular, the duality pairing between $(\W)'$ and $\W$ shall be denoted $\langle \cdot,\cdot\rangle_p$.

By imposing the Robin-type boundary condition $|\grad u|^{p-2}\partial_\nu u + |u|^{p-2}u=0$ on $\Gamma$ the $p$-Laplacian given at the onset of the paper extends readily to a maximal monotone operator from $\W$ into its dual, $(\W)'$, with action given by: 
\begin{align}\label{plapalce}
\langle -\Delta_p u, \phi\rangle_p = \int_\Omega |\grad u|^{p-2}\grad u\cdot\grad\phi\,dx + \int_\Gamma |\g u|^{p-2}\g u \g \phi\,dS,\quad u,\phi\in\W.
\end{align}
Further, it is convenient to record the bound
\begin{align}\label{plaplace-opnorm}
	||-\Delta_p u||_{(\W)'}\leq 2||u||_{1,p}^{p-1}, \quad u\in\W,
\end{align}
on the operator norm of $-\Delta_p u$ which follows easily from H\"older's inequality.  As the Laplacian occurs as a term in equation \eqref{wave}  providing damping, it is efficient to utilize all of the preceding notation formally including the case of $p=2$.  Throughout the paper however, we shall always assume $2<p<3$.   Additionally, the Sobolev embedding (in 3D)
\begin{align*}
\W \into \Lb[\frac{2p}{3-p}]
\end{align*}
as well as the inequalities associated with the trace operator $\gamma$ in the map 
\begin{align*}
\W[1-\epsilon,p]\tinto \Lb[\frac{2p}{3-(1-\epsilon)p}]\into\Lb[4]
\end{align*}
for sufficiently small $\epsilon\geq 0$, will be used frequently  (See, e.g., \cite{ADAMS}).
 As it occurs so frequently we shall pass to subsequences consistently without re-indexing.

\subsection{Main results}
We begin by giving the definition of a weak solution  of \eqref{wave}.
\begin{defn}\label{def:weaksln}  A function $u$ is said to be a weak solution of \eqref{wave} on the interval $[0,T]$ provided:
	\begin{enumerate}[(i)]
		\setlength{\itemsep}{5pt}
		\item\label{def-a} $u\in C_w([0,T];\W)$,
		\item\label{def-b} $u_t\in L^2(0,T;\W[1,2])\cap C_w([0,T];\L[2]),$
		\item\label{def-c} $(u(0),u_t(0))=(u_0,u_1)$ in $\W \times \L[2]$,
		\item\label{def-d} and for all $t\in[0,T]$ the function $u$ verifies the identity
			\begin{align}\label{slnid}
	(u_t(t),\phi(t))_\Omega &- (u_1,\phi(0))_\Omega -\int_0^t(u_t(\tau),\phi_t(\tau))_\Omega\,d\tau \notag \\
		&+ \int_0^t \langle -\Delta_p u(\tau), \phi(\tau)\rangle_p\,d\tau  +\int_0^t \langle -\Delta u_t(\tau), \phi(\tau)\rangle_2 \,d\tau \notag \\ 
		&=\int_0^t\int_\Omega f(u(\tau))\phi(\tau)\,dxd\tau 
		+ \int_0^t \int_\Gamma h(\t u(\tau))\t\phi(\tau)\,dSd\tau 
		\end{align} 	
	\end{enumerate}
for all test functions $\phi\in C_w([0,T];\W)$ with $\phi_t\in L^2(0,T;\W[1,2])$.		
\end{defn}
\begin{rmk}
	In Definition~\ref{def:weaksln} above, $C_w([0,T]; X)$ denotes the space of weakly continuous (often called scalarly continuous) functions from $[0,T]$ into a Banach space $X$.  That is, for each $u\in C_w([0,T];X)$ and  $f \in X'$ the map $t\mapsto \langle f, u(t) \rangle_{X',X}$ is continuous on $[0,T]$.
\end{rmk}
The main results of this work are the following three theorems, the first of which establishes the existence of weak solutions satisfying a suitable energy inequality.
\smallskip
\begin{thm}[\textbf{Local Solutions}]\label{thm:exist}
	Under the stated assumptions, problem \eqref{wave} possesses a local weak solution, $u$, in the sense of Definition~\ref{def:weaksln} on a non-degenerate interval $[0,T]$ with length dependent only upon the initial data, $(u_0,u_1)$, and the local Lipschitz constants of the maps $f:\W\to\L[6/5]$ and $h\circ\t :\W\to\Lb[4/3]$ on a ball about zero of radius prescribed by the initial positive energy. Further, this solution $u$ satisfies the energy inequality 
	\begin{align}\label{energy-2ndid}
	\E(t) + \int_0^t ||u'(\tau)||_{1,2}^2 \,d\tau 
	\leq
	 \E(0) 
	&+ \int_0^t \int_\O f(u(\tau))u'(\tau)\,dx d\tau  \notag \\ 
	&+ \int_0^t \int_\G h(\t u(\tau))\t u'(\tau)\,dSd\tau 
	\end{align}
	where $\E(t)=\frac{1}{2}||u'(t)||_2^2 + \frac{1}{p}||u(t)||_{1,p}^p$ is the positive energy. Equivalently, \eqref{energy-2ndid} can also be written as 
	\begin{align}\label{energy-1stid}
	E(t) + \int_0^t ||u'(\tau)||_{1,2}^2\,d\tau \leq E(0)
	\end{align}
	with $E(t)=\E(t)-\int_\O F(u(t))\,dx - \int_\G H(\t u(t))\,dS$ by taking $F$ and $H$ as the primitives of $f$ and $h$, respectively. i.e., $F(u)=\int_0^u f(s)\,ds$ and $H(\t u)=\int_0^{\t u} h(s)\,ds$. 
\end{thm}

The proof of Theorem~\ref{thm:exist} is carried out in Sections~\ref{S2} 
through \ref{S4}, beginning first with the added assumptions on the source feedback terms $f,h$ and then utilizing a series of truncation arguments similar to \cite{BLR2, BLR1, PRT-p-Laplacain, RW}, amongst others.

\begin{rmk}\label{rmk:embeddings}
	If one chooses the sources $f$ and $h$ of order $q$ and $r$ it follows in accordance with Assumption~\ref{ass:fg} that their primitives $F$ and $H$ are of order $q+1$ and $r+1$ respectively.  From the energy inequality \eqref{energy-1stid}, one would hope to find that the embedding $\W\to\L[q+1]$ along with the trace $\W\tinto\Lb[r+1]$ are both valid in order to ensure integrability of the terms $$\int_\O F(u(t))\,dx\quad\text{ and }\quad\int_\G H(\t u(t))\,dS$$ 
	contained therein.  Indeed, it is readily verified that this is the case given that $2<p<3$.

\end{rmk}

Provided the source terms $f$ and $h$ are of sufficiently small order one would expect the strong damping in $\O$ to produce a global solution which has finite energy on $[0,\infty)$ in line with the results of Webb in \cite{W}, for instance.  This is indeed the case, and the following theorem which is proven in Section~\ref{S5} establishes sufficient growth conditions for a global solution.

\smallskip

\begin{thm}[\textbf{Global Solutions}]\label{thm:global} If $r,q\leq p/2$ in addition to the assumptions of Theorem~\ref{thm:exist}, then the weak solution $u$ furnished by Theorem~\ref{thm:exist} is a global solution and the existence time $T$ may be taken arbitrarily large.
\end{thm}

\begin{rmk} If one were to take $p=2$ the results of Theorem~\ref{thm:global} would state that $f$ and $h$ are bounded by linear functions which is a result paralleling \cite{W} with the addition of boundary terms.  With $p>2$ the sources can be of higher order however, even though the action of the damping term $-\Delta u_t$ is unaffected by this change in $p$.  This peculiar effect has been noted before in \cite{PRT-p-Laplacain}, for instance.
Analogously, in \cite{RW} it was shown that the equation $u_{tt}-\Delta u - \Delta_p u_t=f(u)$ has global solutions when the order of the source term is no more than $p-1$.  In effect, the damping action of $-\Delta_p u_t$ with $2<p<3$ is, in some sense, ``stronger'' than the action of $-\Delta u_t$.
\end{rmk}

Conversely, with sources $f$ and $h$ of sufficient magnitude weak solutions of \eqref{wave} can be shown to have a finite right maximal interval of existence and achieve asymptotically infinite energy in a finite time.  Precisely, assuming the following form of the source functions:

\begin{assumption}\label{ass:blowupsrc}Let $f$ and $h$ be of the form
	\begin{alignat*}{6}
	f(s)&=(q+1)|s|^{q-1}s&\text{ with }&&p-1&<q&<\frac{5p}{2(3-p)},\\
	h(s)&=(r+1)|s|^{r-1}s&\text{ with }&&p-1&<r&<\frac{3p}{2(3-p)}.
	\end{alignat*}
\end{assumption}
Solutions of \eqref{wave} must then blow-up in finite time, stated precisely in the following and proven in Section~\ref{S6}:
\vspace{.1in}
\begin{thm}[\textbf{Blow-up of solutions}]\label{thm:blowup}Assume that $f$ and $h$ are as in Assumption~\ref{ass:blowupsrc} and that the initial data $(u_0,u_1)$ is chosen to have negative total initial energy, in the sense that 	
	\begin{align*}
	E(0)=\E(0)- ||u_0||_{q+1}^{q+1} - |u_0|_{r+1}^{r+1}<0.
	\end{align*} Then, any weak solution $u$ of \eqref{wave} (in the sense of Definition~\ref{def:weaksln}) necessarily blows up in finite time.  That is, there exists some $0<T<\infty$ such that 
	$$\limsup_{t\to T^-}\E(t)=\infty$$
	with positive energy 
	$\E(t)=\frac{1}{2}||u'(t)||_2^2 + \frac{1}{p}||u(t)||_{1,p}^p$  as in Theorem~\ref{thm:exist}.
\end{thm}

\smallskip

Utilizing the Sobolev embedding $\W\to\L[2p/(3-p)]$ in dimension three along with the bounds in Remark~\ref{rmk:fbound} the source feedback term $f$, if regarded as a Nemytski operator $f:\W\to\L[2]$, is seen to be locally Lipschitz continuous provided $1\leq q < 3p/2(3-p)$.  Similarly, the map $h\circ\t:\W\to\Lb[2]$ is locally Lipschitz continuous provided $1\leq r<p/(3-p)$.  A stronger result provided by the following lemma  will be used frequently throughout this paper. Its proof can be found elsewhere, and thus we omit it here.

\begin{lem}[{See  \cite[Lem. 1.4]{MOHNICK1} and
	 \cite[Lem. 1.1]{PRT-p-Laplacain}}]\label{lem:f-lipschitz}
	Under the growth conditions given in Assumption~\ref{ass:fg}, the functions $f:\W[1-\epsilon,p]\to \L[6/5]$ and $h\circ\t :\W[1-\epsilon,p]\to \Lb[4/3]$ are locally Lipschitz continuous for sufficiently small $\epsilon\geq 0$.
\end{lem}

\section{Solutions for globally Lipschitz sources}\label{S2}
As a first step, our strategy is to employ a suitable Galerkin approximation scheme to show local existence of weak solutions of \eqref{wave} in the case where both $f:\W\to\L[2]$ and $h\circ\t:\W\to\Lb[2]$ are globally Lipschitz with constants $L_f$ and $L_h$, respectively.

\subsection{Approximate solutions}\label{S2.1}
To start we shall establish a suitable sequence $\{w_j\}_1^\infty$ with which to construct a sequence of approximate solutions.  This construction is done using eigenfunctions of the Lapalcian--a typical choice for this type of problem--in the following manner.  Let $\mathscr{A}=-\Delta$ with domain $\D(\A)=\{w\in\H[2] : \partial_\nu w + w=0 \text{ on } \Gamma\}\subset\L[2]$.  It is well known that $\mathscr{A}:\D(\A)\subset\L[2]\to\L[2]$ is positive, self-adjoint, and that $\mathscr{A}$ is the inverse of a compact operator. Thus, $\mathscr{A}$ has a countably infinite set of positive eigenvalues $\{\lambda_j\}_{j=1}^\infty$ with $\lambda_j\to\infty$ whose corresponding smooth eigenfunctions $\{w_j\}_{j=1}^\infty$ form an orthonormal basis for $\L[2]$ after suitable normalization. We can also define the fractional powers $\A^s$, $0\leq s\leq 1$, by $\A^sf=\sum_{j=1}^\infty \lambda_j^s(f,w_j)_\O w_j$. Each domain $\D(\A^s)$ is itself a Hilbert space with equivalent inner product 
$$(u,v)_{\D(\A^s)} = \sum_{j=1}^\infty \lambda_j^{2s} u_jv_j$$ for $u=\sum_{j=1}^\infty u_jw_j$ and $v=\sum_{j=1}^\infty v_jw_j$ with convergence in the $\L[2]$ sense. In particular,  the sequence $\{w_j\}_{j=1}^\infty$ forms a Schauder basis for $\D(\A)$.  

Let $V_N= \spn \{w_1,\cdots, w_N\} $ and $\mathscr{P}_N$ be the orthogonal projection of $L^2(\O)$ onto $V_N$.
Corresponding to each $N\in \N$ one may find sequences of scalars $\{u_{N,j}^0\}_{j=1}^\infty$ and $\{u_j^1\}_{j=1}^\infty$ such that 
\begin{subequations}\begin{alignat}{6}\label{approx-ics}
 &\sum_{j=1}^N u^0_{N,j} w_j&\to &u_0&\text{ strongly in }&\W&\text{ as }N\to\infty,\\
\mathscr{P}_Nu_1 =&\sum_{j=1}^N u^1_j w_j&\to &u_1&\text{ strongly in }&\L[2]&\text{ as }N\to\infty
\end{alignat}\end{subequations}
for given initial data $(u_0,u_1)\in\W\times\L[2]$. 
 For the initial displacement $u_0$, these sequences of scalars are obtained using the density of $\spn\{w_j\}_1^\infty$ in $\W$, and  for the initial velocity $u_1$,  $u_j^1=(u_1,w_j)_\O$.

We now seek to construct a sequence of approximate solutions of the form 
\begin{align}\label{approxform}
u_N(x,t)=\sum_{j=1}^N u_{N,j}(t)w_j(x) 
\end{align} 
that satisfies the system: 
\begin{subequations}\label{approxslns}
	\begin{multline}\label{approx1}
	(u_N'',w_j)_\Omega 
	+\overbrace{(|\grad u_N|^{p-2}\grad u_N,\grad w_j)_\Omega + (|\t u_N|^{p-2}\t u_N,\t w_j)_\Gamma}^{\langle -\Delta_p u_N,w_j\rangle_p} \\
	\\
+\underbrace{(\grad u_N',\grad w_j)_\Omega + (\t u_N',\t w_j)_\Gamma}_{\langle -\Delta_2 u_N,w_j\rangle_2 } = (f(u_N),w_j)_\Omega + (h(\t u_N),\t w_j)_\Gamma
	\end{multline} 
	\begin{align}\label{approx2}
	u_{N,j}(0)=u^0_{N,j},\,\,\,u'_{N,j}(0)=u^1_j,
	\end{align}
\end{subequations}
where $j=1,\ldots,N$. 

Indeed,  (\ref{approx1})-(\ref{approx2}) is  an initial value problem for a second order $N\times N$ system of ordinary differential equations with continuous nonlinearities in the unknown functions $u_{N,j}$ and their time derivatives.  Therefore, it follows from the Cauchy-Peano theorem that the initial-value problem \eqref{approxslns} has a solutions $u_{N,j}\in C^2([0,T_N])$, $j=1,\ldots,N$ for some $T_N>0$.

An immediate observation is that the sequence of approximate solutions $\{u_N\}$ satisfies:
\begin{align}\label{2.4}
(u_N(0),u_N'(0))\to (u_0,u_1)\text{ strongly in }\W\times\L[2].
\end{align}

\subsection{A priori estimates} We aim to demonstrate that each of the approximate solutions $u_N$ exists on a non-degenerate interval $[0,T]$ independent of $N$. 
\begin{prop}\label{prop:apriori} Each approximate solution $u_N$ exists on $[0,\infty)$.
Further, for any $0<T<\infty$, the sequence of approximate solutions $\{u_N\}_1^\infty$ satisfies: 
	\begin{subequations}\begin{alignat}{3}
		\{u_N\}_1^\infty&\text{ is a bounded sequence in }&&L^{\infty} (0,T;\W), \label{apriori-a}\\
		\{u_N'\}_1^\infty&\text{ is a bounded sequence in }&&L^{\infty} (0,T;\L[2]),\label{apriori-b}\\
		\{u_N'\}_1^\infty&\text{ is a bounded sequence in }&&L^2(0,T;\H[1]),\label{apriori-c}\\
		\{u_N''\}_1^\infty&\text{ is a bounded sequence in }&&L^2(0,T;(\D(\A))').\label{apriori-d}
		\end{alignat}
	\end{subequations}
\end{prop}
\begin{proof}
Multiplying \eqref{approx1} by $u'_{N,j}$ and summing over $j=1,\ldots,N$, one obtains  
\begin{align}\label{apriori-1}
\frac{1}{2}\frac {d}{dt}||u_N'(\tau)||_2^2 &+ \frac{1}{p}\frac{d}{dt}||u_N(\tau)||_{1,p}^p + ||u_N'(\tau)||_{1,2}^2 \notag \\ &= \int_\Omega f(u_N)u_N'\,dx + \int_ \Gamma  h(\t u_N)\t u_N'\,dS
\end{align}
for each $\tau\in[0,T_N]$.   Integrating \eqref{apriori-1} on $[0,t]$ for $t\in[0,T_N]$ and defining the positive energy
\begin{align*}
	\E_N(t)=\frac{1}{2}||u_N'(t)||_2^2 + \frac{1}{p}||u_N(t)||_{1,p}^p
\end{align*}
we may thus obtain from \eqref{apriori-1} the relation
\begin{multline}\label{apriori-3}
\E_N(t) + \int_0^t ||u_N'(\tau)||_{1,2}^2\,d\tau = \int_0^t\int_\Omega f(u_N(\tau))u_N'(\tau)\,dxd\tau \\+ \int_0^t\int_\Gamma h(\t u_N(\tau))\t u_N'(\tau)\,dSd\tau + \E_N(0).\ptag{apriori-1}
\end{multline}
In order to demonstrate a bound on $\E_N$ we first address the terms due to the sources.   Under the assumption that $f:\W\to\L[2]$ is globally Lipschitz continuous we have 
\begin{align}\label{apriori-3.4}
||f(u_N)||_2 &\leq ||f(u_N)-f(0)||_2 + ||f(0)||_2\notag\\
&\leq L_f||u_N||_{1,p}+||f(0)||_2\notag\\
&\leq C(||u_N||_{1,p}+1),
\end{align}
so that by H\"older and Young's inequalities with $\epsilon=1/4$,

\begin{align}\label{apriori-3.5}
\int_\Omega f(u_N)u_N'\,dx &\leq ||f(u_N)||_2||u_N'||_2\notag\\
&\leq C(||u_N||_{1,p}+1)^2 + \frac{1}{4}||u_N'||_{1,2}^2\notag \\
&\leq C(||u_N||_{1,p}^p+1) + \frac{1}{4}||u_N'||_{1,2}^2,
\end{align}
where we have used the assumption $p>2$. The constant $C$ in \eqref{apriori-3.5} depends upon  the values of $||f(0)||_2$ and $L_f$, but is independent of $N$.  Since $h\circ \t:\W\to\Lb[2]$ is also globally Lipschitz the same argument as in \eqref{apriori-3.4} and \eqref{apriori-3.5} yields 
\begin{align}\label{apriori-3.7}
\int_\G  h(\t u_N)\t u_N'\,dS 
&\leq C(||u_N||_{1,p}^p +1) + \frac{1}{4}||u_N'||_{1,2}^2.
\end{align}
By applying the bounds \eqref{apriori-3.5} and \eqref{apriori-3.7}  to equation \eqref{apriori-3}, we obtain
\begin{align}\label{apriori-4}
\E_N(t)+\frac{1}{2}\int_0^t||u_N'(\tau)||_{1,2}^2\,d\tau &\leq 
C\int_0^t (||u_N(\tau)||_{1,p}^p+1)\,d\tau + \E_N(0).
\end{align}
Recalling \eqref{2.4}, we see the that the sequence  $\{\E_N(0) = \frac{1}{2}||u_N'(0)||_2^2 + \frac{1}{p}||u_N(0)||_{1,p}^p\}$ 
is  bounded, say $\E_N(0)\leq C'$ for all $N$. In addition, $||u_N(\tau)||_{1,p}^p<p\E_N(\tau)$  so that each $\E_N(t)$ thereby satisfies the integral inequality 
\begin{align}\label{apriori-4.5}
\E_N(t)\leq C\int_0^t \left( \E_N(\tau) +1 \right)\,d\tau + C'
\end{align} 
for some positive constants $C, C'$ independent of $N$. By applying Gronwall's inequality to \eqref{apriori-4.5} we see that $\E_N(t)$ is finite on $[0,T]$ for \emph{any} $0<T<\infty$, and upon this interval each $u_N$ must exist by the Cauchy-Peano theorem. This bound on $\E_N(t)$ also establishes \eqref{apriori-a} and \eqref{apriori-b}, with \eqref{apriori-c} following immediately from \eqref{apriori-4} whose right hand side is now seen to be bounded.

For the final claim, given any $\phi\in\D(\A)$ we know that $\phi$ has a unique expansion in terms of the Schauder basis $\{w_j\}_1^\infty$ for $\D(\A)$ as 
$$\phi=\sum_{j=1}^\infty a_jw_j,$$
for scalars $\{a_j\}_1^\infty\subset \R$. Furthermore, by taking $$S_N\phi=\sum_{j=1}^Na_jw_j$$ to be the canonical projection associated with the basis $\{w_j\}_1^\infty$ we know that $S_N\phi\to \phi$ strongly in $\D(\A)$ and also that $||S_N\phi||_{\D(\A)}\leq \mathcal{C}||\phi||_{\D(\A)}$ for a finite basis constant $1\leq \mathcal{C}<\infty$.  Thus, from \eqref{approx1} and the orthogonality of the sequence $\{w_j\}$ in $\L[2]$ we find with $\langle\cdot,\cdot\rangle$ denoting the pairing between $\D(\A)$ and its dual that 
\begin{align*}
|\langle u_N'',\phi\rangle|  %
&=|(u_N'',S_N\phi)_\O|\notag \\
&\leq |\langle -\Delta_p u_N, S_N\phi\rangle_p| +| \langle -\Delta_2 u_N',S_N\phi\rangle_2| \\
&\qquad
+ |(f(u_N),S_N\phi)_\Omega| + |(h(\t u_N),\t S_N\phi)_\Gamma|\\
&\leq 2||u_N(t)||_{1,p}^{p-1}||S_N\phi||_{1,p} + 2||u_N'(t)||_{1,2}||S_N\phi||_{1,2} \notag \\ 
&\qquad + ||f(u_N(t))||_2||S_N\phi||_2 + |h(\t u_N(t))|_2|\t S_N\phi|_2\\
&\leq C\Big(||u_N(t)||_{1,p}^{p-1} + ||u_N'(t)||_{1,2} + ||f(u_N(t))||_2 + |h(\t u_N(t))|_2 \Big)||S_N\phi||_{\D(\A)}\notag\\ 
&\leq C\Big(||u_N(t)||_{1,p}^{p-1} + ||u_N(t)||_{1,p}+  ||u_N'(t)||_{1,2} + 1\Big)||\phi||_{\D(\A)}\notag
\end{align*}
by utilizing the operator norm bound on $-\Delta_p$ along with the bounds on $||f(u_N(t))||_2$ and $|h(\t u_N(t))|_2$ from the argument in \eqref{apriori-3.4}.  Since we have already shown that $||u_N(t)||_{1,p}$ is bounded on $[0,T]$ from \eqref{apriori-a} we find that  
$$|\langle u_N'',\phi\rangle| \leq C(1+||u_N'(t)||_{1,2} )||\phi||_{\D(\A)}.$$ 
Thus, as $||u_N'(t)||_{1,2}\in L^2(0,T)$ from \eqref{apriori-c} it follows that $||u_N''(t)||_{(\D(\A))'}\in L^2(0,T)$, completing the proof.
\end{proof}

An immediate consequence of Proposition~\ref{prop:apriori} along with the Banach-Alaoglu theorem and the standard Aubin-Lions-Simon compactness theorems (e.g., \cite[Thm. II.5.16]{Boyer2013}) is the following:
 \begin{subequations}
\begin{cor}\label{cor:converg} For all sufficiently small $\epsilon>0$ there exists a function $u$ and a subsequence of $\{u_N\}$ (still denoted by $\{u_N\}$) such that  
	 \begin{alignat}{2}
	 	u_N&\to u &\text{ weak* in }&L^\infty(0,T;\W), \label{converg:a} \\
		u_N'&\to u' &\text{ weak* in }&L^\infty(0,T;\L[2]),\label{converg:b} \\
		u_N'&\to u' &\text{ weakly in }&L^2(0,T;\H[1]),\label{converg:c} \\
		u_N &\to u &\text{ strongly in }&C([0,T];\W[1-\epsilon,p]),\label{converg:d} \\
		u_N' &\to u' &\text{ strongly in }&L^2(0,T;\W[1-\epsilon,2]),\label{converg:e}
	\end{alignat}
\end{cor}

 By utilizing a routine density argument we also obtain convergence in the following sense:
\begin{cor}\label{cor:aeinw1p}
	On a subsequence,
	\begin{alignat}{2}
	u_N(t)&\to u(t)&\text{ weakly in }&\W\text{ for a.e. }t\in[0,T],\label{converg:f}\\
	u_N(t) &\to u(t) & \text{ weakly in }& \W[1,2]\text{ for a.e. } t\in[0,T]. \label{converg:g}
	\end{alignat}
\end{cor}
\begin{proof}
	The result follows immediately from \cite[Proposition~A.2]{PRT-p-Laplacain} as a consequence of \eqref{apriori-a} and \eqref{converg:d}.
\end{proof}
\end{subequations}

\subsection{Passage to the limit}
By integrating \eqref{approx1} on $[0,t]$ it is seen that each approximate solution $u_N$ verifies the identity 
\begin{multline}\label{newnewlim1}
(u_N'(t),w_j)_\O - (u_N'(0),w_j)_\O 
 + \int_0^t \langle -\Delta_p u_N(\tau),w_j\rangle_p\,d\tau
+ \int_0^t \langle -\Delta_2 u_N'(\tau),w_j\rangle_2\,d\tau\\
= \int_0^t \int_\O f(u_N(\tau))w_j\,dxd\tau 
+ \int_0^t \int_\G h(\t u_N(\tau))\t w_j\,dSd\tau.
\end{multline}
for $j=1,\ldots,N$. As a first step in demonstrating that the limit function $u$ indeed verifies the variational identity \eqref{slnid} we shall first carefully pass to the limit as $N\to\infty$ in \eqref{newnewlim1}.  For most of the terms except for the $p$-Laplacian this will be routine, and this convergence is addressed in the following propositions.
\begin{prop}\label{prop:limf}
	\begin{subequations}\begin{alignat}{3}
	f(u_N)&\to f(u)&\text{ strongly in }&L^\infty(0,T;\L[6/5]),\label{prop:limf-a}\\
	h(\t u_N)&\to h(\t u)&\text{ strongly in }&L^\infty(0,T;\Lb[4/3]).\label{prop:limf-b}
	\end{alignat}\end{subequations}
\end{prop}
\begin{proof}
	From the convergence in \eqref{converg:d} we find that $$||u_N(t)||_{1,p},||u(t)||_{1,p} \leq R, \,\,\, t\in[0,T]$$ for a constant $R>0$ independent of $N$. Thus, from Lemma~\ref{lem:f-lipschitz} and \eqref{converg:d} we obtain 
	\begin{align*}
	||f(u_N(t))-f(u(t))||_{6/5} \leq C_R||u_N(t)-u(t)||_{1-\epsilon,p}\to 0, \,\,\,\, t\in[0,T],\\ 
	|h(\t u_N(t)) - h(\t u(t))|_{4/3}\leq C_R||u_N(t)-u(t)||_{1-\epsilon,p}\to 0,  \,\,\,\, t\in[0,T],
	\end{align*}
completing the proof. 	
\end{proof}

\begin{prop}\label{prop:ut}With $\{u_N\}$ and $u$ as in Corollary~\ref{cor:converg},
	$$-\Delta_2 u_N' \to -\Delta_2 u'\text{ weakly in }L^2(0,T;(\W[1,2])').$$
\end{prop}
\begin{proof}
For $\phi\in L^2(0,T;\W[1,2])$ we have 
\begin{align*}
\int_0^T \langle -\Delta_2 u_N'(\tau),\phi(\tau)\rangle_2\,d\tau = 
\underbrace{\int_0^T \int_\O \grad u_N'(\tau)\cdot\grad\phi(\tau)\,dxd\tau}_\text{(i)} + 
\underbrace{\int_0^T \int_\G \t u_N'(\tau)\t \phi(\tau)\,dSd\tau}_\text{(ii)}.	
\end{align*}
For $\text{(i)}$, we have $\grad u_N' \to \grad u'$ weakly in $L^2(0,T;L^2(\O))$ from \eqref{converg:c} with $\grad\phi\in L^2(0,T;L^2(\O))$; and for $\text{(ii)}$, we use the fact that $\t u_N' \to \t u'$ strongly in $L^2(0,T;\Lb[2])$ from \eqref{converg:e} along with the continuity of the map $\W[1-\epsilon,2]\tinto \Lb[2]$ for sufficiently small $\epsilon>0$.  Thus,
\begin{align*}
\int_0^T \langle -\Delta_2 u_N'(\tau),\phi(\tau)\rangle_2 \,d\tau &\to  
\int_0^T \int_\O \grad u'(\tau)\cdot\grad\phi(\tau)\,dxd\tau
+ \int_0^T \int_\G \t u'(\tau)\t \phi(\tau)\,dSd\tau\\
&= \int_0^T \langle -\Delta_2 u'(\tau),\phi(\tau)\rangle_2\,d\tau,
\end{align*}
completing the proof. 
\end{proof}

We are now in a position to address the far more delicate matter of the convergence of the term arising from the $p$-Laplacian.  The nonlinearity of this operator is one of the primary challenges in this process, and the argument via monotone operator theory is necessarily detailed.

\begin{prop}\label{prop:limlaplace}
On a subsequence, the sequence $\{u_N\}$ and the limit function $u$ from Corollary~\ref{cor:converg} satisfy 
\begin{align*}
-\Delta_p u_N \to -\Delta_p u\text{ weak* in }L^\infty(0,T;(\W)').
\end{align*}
\end{prop}
\begin{proof}
	Throughout, set $X=L^p(0,T;\W)$. By utilizing the operator norm bound on $-\Delta_p$ from (\ref{plaplace-opnorm}) we see that 
	\begin{align*}
	||-\Delta_p u_N||_{L^\infty(0,T;(\W)')} &= \esssup_{\tau\in[0,T]} ||-\Delta_p u_N (\tau)||_{(\W)'}\\
	&\leq 2\esssup_{\tau\in[0,T]}||u_N(\tau)||_{1,p}^{p-1}\leq C
	\end{align*}
	for a constant $C$ independent of $N$ by virtue of \eqref{converg:a}, whereupon the sequence $\{-\Delta_p u_N\}$ is seen to be bounded in $L^\infty(0,T;(\W)')$. As such, there exists by the Banach-Alaoglu theorem some $\eta\in L^\infty(0,T;(\W)')$ and a subsequence of $\{u_N\}$ so that 
	\begin{align}\label{limlaplace-1}
	-\Delta_p u_N \to \eta\text{ weak* in }L^\infty(0,T;(\W)').
	\end{align}
	By viewing $L^\infty(0,T;(\W)')$ as a subspace of $X'$ the desired conclusion follows immediately by demonstrating that $-\Delta_p u_N \to -\Delta_p u$ weakly in $X'$ in keeping in line with a standard result from analysis  (e.g., \cite[Prop II.2.10]{Boyer2013}).  Towards these ends, the operator $-\Delta_p$ extends to a maximal monotone operator from $X$ into $X'$ as   
	\begin{align*}
	\langle -\Delta_p u, \phi\rangle_{X',X}=\int_0^T\langle -\Delta_p u(\tau),\phi(\tau)\rangle_p \,d\tau;\quad u,\phi\in X,
	\end{align*}
	with an elementary proof provided in  \cite[Lem. 5.1]{MOHNICK1}.
	
	In order to conclude that $\eta=-\Delta_p u$ in $X'$ we appeal to a standard result from monotone operator theory (see \cite[Cor. 2.4]{Barbu2010}, for instance) and demonstrate that 
	\begin{align}\label{limlaplace-wts}
	\limsup_{N\to\infty}\langle -\Delta_p u_N,u_N\rangle_{X',X} \leq \langle  \eta, u\rangle_{X',X}.
	\end{align}
	
	Multiplying equation \eqref{approx1} by $u_{N,j}$ and summing over $j=1,\ldots,N$ we obtain the relation 
	\begin{multline}\label{limlaplace-s1}
	(u_N'',u_N)_\Omega + \langle -\Delta_p u_N,u_N\rangle_p + \overbrace{(\grad u_N',\grad u_N)_\Omega + (\t u_N',\t u_N)_\Gamma}^{\langle -\Delta_2 u_N',u_N\rangle_2} \\ \vspace{.1in} = (f(u_N),u_N)_\O + (h(\t u_N),\t u_N)_\Gamma 
	\end{multline}
	using the same summation relations as were demonstrated at the onset of the proof of Proposition~\ref{prop:apriori}. Rearranging \eqref{limlaplace-s1} and integrating over $[0,t]$ we thus obtain 
	\begin{multline}\label{limlaplace-s1p}
	\int_0^t \langle -\Delta_p u_N,u_N\rangle_p \,d\tau = -\int_0^t (u_N'',u_N)_\Omega\,d\tau -\int_0^t (\grad u_N',\grad u_N)_\Omega\,d\tau \\
	-\int_0^t (\t u_N',\t u_N)_\Gamma \,d\tau +\int_0^t \int_\O f(u_N)u_N \,dx \,d\tau + \int_0^t \int_\G h(\t u_N)\t u_N\,dSd\tau. \ptag{limlaplace-s1}
	\end{multline}
	Thus, upon integrating by parts and making the identification 
	$$(\grad u_N',\grad u_N)_\O + (\t u_N'.\t u_N)_\G =\frac{1}{2}\frac{d}{d t}||u_N||_{1,2}^2
	$$
	we may write 
	\begin{multline}\label{limlaplace-s1pp}
	\int_0^t \langle -\Delta_p u_N,u_N\rangle_p \,d\tau = 
	\underbrace{(u_N'(0),u_N(0))_\Omega - (u_N'(t),u_N(t))_\Omega}_\text{(i)}  \\
	+ \underbrace{\int_0^t ||u_N'(\tau)||_2^2\,d\tau}_\text{(ii)} 
	+ \underbrace{\frac{1}{2}||u_N(0)||_{1,2}^2}_\text{(iii)}
	-\underbrace{\frac{1}{2}||u_N(t)||_{1,2}^2}_\text{(iv)}\\
	+\underbrace{\int_0^t \int_\O f(u_N)u_N \,dx\,d\tau 
	+ \int_0^t \int_\G h(\t u_N)\t u_N\,dSd\tau}_\text{(v)}. \ptag{limlaplace-s1p}
	\end{multline}
	The convergence of these terms warrants special attention:
	\begin{enumerate}[(i)]
		\setlength{\itemsep}{5pt}
		\item From (\ref{2.4}),  $(u_N'(0),u_N(0))_\Omega\to(u_1,u_0)_\Omega$. Using \eqref{converg:d} in Corollary~\ref{cor:converg} we obtain $||u_N-u||_2\to 0$ in $L^2(0,T)$, and hence on a subsequence $u_N(t)\to u(t)$ strongly in $\L[2]$ for a.e. $t\in[0,T]$. Similarly, from \eqref{converg:e} we find $u_N'(t)\to u'(t)$ strongly in $\L[2]$ for a.e. $t\in[0,T]$ on a subsequence. Thus, 
		$$(u_N'(t),u_N(t))_\O \to (u'(t),u(t))_\O\text{ a.e. }t\in[0,T]$$
		on a common subsequence of $\{u_N\}$ and $\{u_N'\}$.		
		
		\item Since $u_N'\to u'$ strongly in $L^2(0,T;\L[2])$ from \eqref{converg:e} in Corollary~\ref{cor:converg}, it follows that  
		\[ \int_0^t ||u_N'(\tau)||_2^2\,d\tau \to \int_0^t ||u'(\tau)||_2^2\,d\tau.
		\]

		\item From the convergence in (\ref{2.4}) and the embedding $\W\into\W[1,2]$ we obtain $$ \frac{1}{2}||u_N(0)||_{1,2}^2 \to \frac{1}{2}||u(0)||_{1,2}^2.$$
		
		\item From Corollary~\ref{cor:aeinw1p} it follows that $u_N(t)\to u(t)$ weakly in $\W[1,2]$ for a.e. $t\in[0,T]$.  Using the weak lower-semicontinuity of norms we obtain  
		$$ \limsup_{N\to\infty} -\frac{1}{2}||u_N(t)||_{1,2}^2 = 
		-\frac{1}{2}\liminf_{N\to\infty} ||u_N(t)||_{1,2}^2 \leq -\frac{1}{2} ||u(t)||_{1,2}^2\quad \text{a.e. }[0,T].$$
		 
		 \item Since $f(u_N)\to f(u)$ strongly in $L^\infty(0,T;\L[6/5])$ from Proposition~\ref{prop:limf} and $u_N\to u$ strongly in $C([0,T];\L[6])$ from \eqref{converg:d} in Corollary~\ref{cor:converg} and the Sobolev embedding $\W[1-\epsilon,p]\into \L[6]$, 
		 \[
		 \int_0^t \int_\O f(u_N)u_N\,dxd\tau \to \int_0^t\int_\O f(u)u\,dxd\tau
		 \]
		 for $t\in[0,T]$.  Similarly, since $h(\t u_N)\to h(\t u)$ strongly in $L^\infty(0,T;\Lb[4/3])$ and $\t u_N \to \t u$ strongly in $C([0,T];\Lb[4])$ from \eqref{converg:d} and the trace $\W\tinto \Lb[4]$, 
		 \[
		 \int_0^t \int_\G h(\t u_N)\t u_N\,dSd\tau \to \int_0^t \int_\G h(\t u)\t u\,dSd\tau
		 \]
		 for $t\in[0,T]$.
	\end{enumerate}
	We may thus take the limit superior as $N\to\infty$ in \eqref{limlaplace-s1pp} to obtain 
	\begin{multline}\label{limlaplace-s2}
	\limsup_{N\to\infty} \int_0^t  \langle -\Delta_p u_N,u_N\rangle_p\,d\tau 
	\leq (u'(0),u(0))_\Omega - (u'(t),u(t))_\Omega  \\ 
	+ \int_0^t ||u'(\tau)||_2^2\,d\tau 
	+ \frac{1}{2}||u(0)||_{1,2}^2 
	- \frac{1}{2}||u(t)||_{1,2}^2 \\
	+\int_0^t \int_\O f(u)u\,dxd\tau  
	+ \int_0^t \int_\G h(\t u)\t u \,dS \,d\tau \quad \text{ a.e. }[0,T].
	\end{multline}	
	In order to express the right hand side of \eqref{limlaplace-s2} in terms of $\eta$ we utilize the separable nature of the approximate solutions to effect a limit of \eqref{limlaplace-s1pp} through a different means. Towards these ends, multiplying \eqref{approx1} by any $\phi\in C^1([0,T])$ and integrating on $[0,t]$ yields 
	\begin{multline}\label{limlaplace-s3}
	\int_0^t \langle -\Delta_p u_N,\phi w_j\rangle_p \,d\tau = 
	(u_N'(0),\phi(0)w_j)_\Omega - (u_N'(t),\phi(t)w_j)_\Omega \\ 
	+ \int_0^t (u_N',\phi'w_j)_\Omega\,d\tau 
	-\int_0^t (\grad u_N',\phi\grad w_j)_\Omega\,d\tau  
	-\int_0^t (\t u_N',\phi \t w_j)_\Gamma \,d\tau \\
	+\int_0^t \int_\O f(u_N)\phi w_j\, dx \,d\tau + \int_0^t \int_\G h(\t u_N)\phi\t w_j\,dSd\tau.
	\end{multline}
	Taking the limit as $N\to\infty$ in \eqref{limlaplace-s3} is readily justified in each of the terms from the weak convergence given in Corollary~\ref{cor:converg}.  It is perhaps worthy of note however that 
	\begin{align*}\int_0^t (\grad u_N',\phi\grad w_j)_\O\,d\tau + \int_0^t (\t u_N',\phi\t w_j)_\G\,d\tau &=  \int_0^t \langle -\Delta_2 u_N',\phi w_j\rangle_2\,d\tau \\
	 &\to \int_0^t \langle -\Delta_2 u',\phi w_j\rangle_p\,d\tau\end{align*}
	from Proposition~\ref{prop:ut} since $\phi w_j\in L^2(0,T;\W[1,2])$.	
	Thus,
	\begin{multline}\label{limlaplace-s4}
	\int_0^t \langle \eta,\phi w_j\rangle_p \,d\tau = 
	(u'(0),\phi(0)w_j)_\Omega - (u'(t),\phi(t)w_j)_\Omega \\ 
	+ \int_0^t (u',\phi'w_j)_\Omega\,d\tau 
	-\int_0^t (\grad u',\phi\grad w_j)_\Omega\,d\tau  
	-\int_0^t (\t u',\phi \t w_j)_\Gamma \,d\tau \\
	+\int_0^t \int_\O f(u)\phi w_j\, dx \,d\tau + \int_0^t \int_\G h(\t u)\phi\t w_j\,dSd\tau.
	\end{multline}
	The identification
	\begin{align*}
	\lim_{N\to\infty}\int_0^t \langle -\Delta_p u_N,\phi w_j\rangle_p \,d\tau = \int_0^t \langle \eta , \phi w_j\rangle_p \,d\tau
	\end{align*}
	in this limit is possible since $\-\Delta_p u_N\to\eta$ in $X'$ and $\phi w_j\in X$ Now, replacing $\phi(t)$ with $u_{N,j}(t)$ in \eqref{limlaplace-s4}  and summing over $j=1,\ldots,N$ we obtain  
	\begin{multline}\label{limlaplace-s5}
	\int_0^t \langle \eta,u_N\rangle_p \,d\tau = 
	(u'(0),u_N(0))_\Omega - (u'(t),u_N(t))_\Omega \\ 
	+ \int_0^t (u',u_N')_\Omega\,d\tau 
	-\int_0^t (\grad u',\grad u_N)_\Omega\,d\tau  
	-\int_0^t (\t u',\t u_N)_\Gamma \,d\tau \\
	+\int_0^t \int_\O f(u)u_N\, dx \,d\tau + \int_0^t \int_\G h(\t u)\t u_N\,dSd\tau.	\end{multline}
	Taking the limit in \eqref{limlaplace-s5} as $N\to\infty$ we obtain
	\begin{multline}\label{limlaplace-s6}
	\int_0^t \langle \eta,u_N\rangle_p \,d\tau = 
	(u'(0),u(0))_\Omega - (u'(t),u(t))_\Omega \\ 
	+ \int_0^t (u',u')_\Omega\,d\tau 
	-\int_0^t (\grad u',\grad u)_\Omega\,d\tau  
	-\int_0^t (\t u',\t u)_\Gamma \,d\tau \\
	+\int_0^t \int_\O f(u)u\, dx \,d\tau + \int_0^t \int_\G h(\t u)\t u\,dSd\tau,	\end{multline}
	whose right hand side is identical to \eqref{limlaplace-s2} after identifying
	\[\int_0^t (\grad u',\grad u)_\O\,d\tau + \int_0^t (\t u',\t u)_\G\,d\tau  = \frac{1}{2}||u(t)||_{1,2}^2 - \frac{1}{2}||u(0)||_{1,2}^2
	\]
	with the aid of \cite[Proposition II.5.11]{Boyer2013}, for instance. That is, we have shown
	\begin{align*}
	\limsup_{N\to\infty}\int_0^t \langle -\Delta_p u_N, u_N\rangle_p\,d\tau \leq \int_0^t \langle \eta,u\rangle_p\,d\tau\text{ a.e. }[0,T].
	\end{align*}
	Hence,  \eqref{limlaplace-wts} is indeed valid and we have $-\Delta_p u_N\to -\Delta_p u$ weakly in $X'$ completing the proof.
\end{proof}

With the aid of Propositions~\ref{prop:limf}, \ref{prop:ut}, and \ref{prop:limlaplace} we are now justified in taking the limit in \eqref{newnewlim1} and concluding that the limit function $u$ satisfies the identity 
\begin{multline}\label{newnewlim2}
(u'(t),w_j)_\O - (u'(0),w_j)_\O
+ \int_0^t \langle -\Delta_p u(\tau),w_j\rangle_p\,d\tau
+ \int_0^t \langle -\Delta_2 u'(\tau),w_j\rangle_2\,d\tau\\
= \int_0^t \int_\O f(u(\tau))w_j\,dxd\tau 
+ \int_0^t \int_\G h(\t u(\tau))\t w_j\,dSd\tau.
\end{multline} 
for all $j\in\mathbb{N}$ and a.e. $t\in[0,T]$.  From the density of $\spn\{w_j\}_{j=1}^\infty$ in $\W$ we thus obtain 
\begin{multline}\label{newnewlim3}
(u'(t),\psi)_\O - (u'(0),\psi)_\O
+ \int_0^t \langle -\Delta_p u(\tau),\psi\rangle_p\,d\tau
+ \int_0^t \langle -\Delta_2 u'(\tau),\psi\rangle_2\,d\tau\\
= \int_0^t \int_\O f(u(\tau))\psi\,dxd\tau 
+ \int_0^t \int_\G h(\t u(\tau))\t \psi\,dSd\tau.
\end{multline} 
for all $\psi\in\W$ and a.e. $t\in[0,T]$.
Before proceeding, we pause to verify that $u''$ has the desired additional regularity.
\begin{lem}\label{lem:utt}
The limit function $u$ identified in Corollary~\eqref{cor:converg} verifying identity \eqref{newnewlim3} satisfies $u''\in L^2(0,T;(\W)')$.
\end{lem}
\begin{proof}
	Given any $\phi\in\W$ we obtain from \eqref{newnewlim3} that 
	\begin{multline*}
	\langle u'(t),\phi\rangle_p = (u'(t),\phi)_\O = (u'(0),\phi)_\O - \int_0^t \langle -\Delta_p u(\tau),\phi\rangle_p\,d\tau - \int_0^t \langle -\Delta_2 u'(\tau),\phi\rangle_2\,d\tau \\
	+ \int_0^t\int_\O f(u(\tau))\phi\,dxd\tau + \int_0^t \int_\G h(\t u(\tau))\t \phi\,dSd\tau,
	\end{multline*}
	wherein it is clear that $\langle u'(t),\phi\rangle_p$ is an absolutely continuous function on $[0,T]$ with 
	\begin{align*}
	\left| \frac{d}{dt}\langle u'(t),\phi\rangle_p \right| 
	\leq & |\langle -\Delta_p u(t),\phi	\rangle_p| + |\langle -\Delta_2 u'(t),\phi\rangle_2| \\
	&+ \int_\O |f(u(t))\phi|\,dx 
	+ \int_\G |h(\t u(t))||\t\phi|\,dS, \text{    a.e.     } [0,T].
	\end{align*}
	Each of these terms is readily bounded by the operator norm bounds on $-\Delta_p$ and the bounds on $f$ and $h$ exactly as was done at the end of the proof of Proposition~\ref{prop:apriori}. Thus, we find
	\begin{align}\label{utt-bound}
	|\langle u''(t),\phi\rangle_p| &\leq C \Big(||u(t)||_{1,p}^{p-1} + ||u'(t)||_{1,2} + ||u(t)||_{1,p}+1 \Big) 
	||\phi||_{1,p}\notag \\
	&\leq C(||u'(t)||_{1,2} +1 )||\phi||_{1,p}
	\end{align}
	since $||u(t)||_{1,p}$ is bounded a.e. $[0,T]$ from \eqref{converg:a}. The desired result then follows by integrating the square of \eqref{utt-bound} on $[0,T]$ since $||u'(t)||_{1,2}\in L^2(0,T)$ by \eqref{converg:c}. 
\end{proof}
\medskip

\subsection{Verification that the limit is a solution}
To verify that the limit function $u$ given in Corollary~\ref{cor:converg} does indeed satisfy every criterion of Definition~\ref{def:weaksln} we begin by recording its regularity in time, which is an immediate consequence of a well-known result by often attributed to Lions and Magenes as in \cite[Lem. 8.1]{LM1} and given here without proof. 

\begin{cor}\label{cor:uinCw}
	Up to possible modification on a set of measure zero, the limit function $u$ and its derivative $u'$ identified in Corollary~\ref{cor:converg} satisfy the additional regularity:
	$$u\in C_w([0,T];\W) \text{    \, and \,        } u'\in C_w([0,T];\L[2]).$$
\end{cor}

We now must show that the limit function $u$ satisfies the variational identity \eqref{slnid} which permits time dependent test functions.  Through a density arguemnt as in \cite[Prop. A.1]{PRT-p-Laplacain} it can be shown that the regularity afforded by Lemma~\ref{lem:utt} implies a product rule of the form 
\begin{align}\label{prodrule}
	\frac{d}{dt}(u'(\tau),\phi(\tau))_\O = \langle u''(\tau),\phi(\tau)\rangle_p + (u'(\tau),\phi'(\tau))_\O
\end{align}
is indeed valid for any test function $\phi\in C_w([0,T];\W)$ with $\phi'\in L^2(0,T;\W[1,2])$.  With this we may express \eqref{newnewlim3} equivalently as 
\begin{multline}\label{passtolim1}
\int_0^t \langle u''(\tau),\psi\rangle_p\,d\tau 
+ \int_0^t \langle -\Delta_p u(\tau),\psi\rangle_p\,d\tau
+ \int_0^t \langle -\Delta_2 u'(\tau),\psi\rangle_2\,d\tau\\
= \int_0^t \int_\O f(u(\tau))\psi\,dxd\tau 
+ \int_0^t \int_\G h(\t u(\tau))\t \psi\,dSd\tau.
\end{multline} 
As each term in \eqref{passtolim1} is absolutely continuous we may differentiate in time and then replace $\psi$ with $\phi(\tau)$ for any time-dependent test function $\phi$. Integrating on $[0,t]$ and again utilizing the product rule \eqref{prodrule} we obtain the desired identity,
\begin{multline*}%
\overbrace{(u_t(t),\phi(t))_\O - (u_1,\phi(0))_\O - \int_0^t (u'(\tau),\phi'(\tau))_\O\,d\tau}^{\int_0^t \langle u''(\tau),\phi(\tau)\rangle_p\,d\tau}  
+ \int_0^t \langle -\Delta_p u(\tau),\phi(\tau)\rangle_p\,d\tau\\
+ \int_0^t \langle -\Delta_2 u'(\tau),\phi(\tau)\rangle_2\,d\tau
= \int_0^t \int_\O f(u(\tau))\phi(\tau)\,dxd\tau 
+ \int_0^t \int_\G h(\t u(\tau))\t \phi(\tau)\,dSd\tau.
\end{multline*} 
This completes the proof that the limit function $u$ is indeed a solution in every sense of Definition~\ref{def:weaksln}.

\subsection{Energy inequality}
In order to complete the proof of Theorem~\ref{thm:exist}  in the case where $f,h$ are globally Lipschitz it remains only to establish the appropriate energy inequalities which are given in the following proposition.

\begin{prop}\label{prop:energy} The limit function $u$ identified in Corollary~\ref{cor:converg} satisfies the energy inequalities \eqref{energy-2ndid} and \eqref{energy-1stid} in the statement of Theorem~\ref{thm:exist}.
\end{prop}
\begin{proof} 
From \eqref{apriori-3} in the course of establishing the a priori estimates it was shown that each $u_N$ satisfies 
\begin{multline}\label{energy-2}
\E_N(t) + \int_0^t ||u_N'(\tau)||_{1,2}^2\,d\tau = \int_0^t\int_\Omega f(u_N(\tau))u_N'(\tau)\,dxd\tau \\+ \int_0^t\int_\Gamma h(\t u_N(\tau))\t u_N'(\tau)\,dSd\tau + \E_N(0)
\end{multline}
with positive energy $\E_N (t)= \frac{1}{2}||u_N'(t)||_2^2 + \frac{1}{p}||u_N(t)||_{1,p}^p$, so that 
\begin{multline}\label{energy-3}
\E_N(t) + \int_0^t ||u_N'(\tau)||_{1,2}^2\,d\tau = \int_\O \left( F(u_N(t)) - F(u_N(0))\right)\,dx \\
+ \int_\G \left( H(\t u_N(t)) - H(\t u_N(0))\right)\,dS + \E_N(0)
\end{multline}
by taking $F(u_N)=\int_0^{u_N} f(s)\,ds$ and $H(\t u_N) = \int_0^{\t u_N}h(s)\,ds$ as the primitives of $f$ and $h$. 
By defining the total energy
\begin{align*}
 E_N(t) = \E_N (t) - \int_\O F(u_N(t))\,dx - \int_\G H(\t u_N(t))\,dS
\end{align*}
we may then re-express \eqref{energy-3} as 
\begin{align}\label{energy-4}
E_N(t) + \int_0^t ||u_N'(\tau)||_{1,2}^2 \,d\tau = E_N(0)\ptag{energy-3}.
\end{align}
In order to pass to the limit in \eqref{energy-4} and establish \eqref{energy-2ndid} we first address the convergence of the terms arising from $F$ and $H$. From the mean value theorem for integrals, 
\begin{align*}
|F(u_N)-F(u)| = |f(\xi)||u_N-u| \leq C (|\xi|^q+1)|u_N-u|
\end{align*}
for some $\xi$ with $|\xi|\leq |u_N|+|u|$. Thus, by using the bounds in Remark \ref{rmk:fbound} along with H\"older's inequality with conjugate exponents $6$ and $6/5$ we then have 
\begin{align}\label{energy-5a}
\int_\O |F(u_N)-F(u)|\,dx &\leq C \int_\O (1+|u_N|^q + |u|^q )|u_N-u|\,dx\notag \\
& \leq C\left(\int_\O (1+|u_N|^q+|u|^q)^{6/5}\,dx\right)^{5/6} ||u_N-u||_6\notag \\
&\leq C(1+||u_N||_{6q/5}^q+||u||_{6q/5}^q)||u_N-u||_6\notag \\
& \leq C(1+||u_N||_{1-\epsilon,p}^q + ||u||_{1-\epsilon,p}^q)||u_N-u||_{1-\epsilon,p}
\end{align}
from the continuity of the embeddings $\W[1-\epsilon,p]\into \L[6q/5]$ and $\W[1-\epsilon,p]\into\L[6]$. Analogously, 
\begin{align}\label{energy-5b}
\int_\G |H(\t u_N)-H(\t u)|\,dS 
& \leq C(1+||u_N||_{1-\epsilon,p}^r + ||u||_{1-\epsilon,p}^r)||u_N-u||_{1-\epsilon,p}
\end{align}
from the continuity of the maps $\W[1-\epsilon,p]\tinto \Lb[4r/3]$ and $\W[1-\epsilon,p]\tinto \Lb[4]$.
From the convergence in \eqref{converg:d} we thus obtain 
\begin{multline}\label{energy-6}
\lim_{N\to\infty} \left(\int_\O F(u_N(t))\,dx + \int_\G H(\t u_N(t)))\,dS\right) \\= \int_\O F(u(t))\,dx + \int_\G H(\t u(t))\,dS;\quad t\in[0,T].
\end{multline}
By recalling the convergence in \eqref{converg:c} we may thus utilize weak lower semicontinuity of the norms to conclude from \eqref{energy-6} and \eqref{energy-4}, respectively, that 
\begin{align}\label{energy-7}
E(t) + \int_0^t ||u'(\tau)||_{1,2}^2\,d\tau &\leq \liminf_{N\to\infty}\left( E_N(t) + \int_0^t ||u_N'(\tau)||_{1,2}^2 \,d\tau \right)\notag \\
&= \liminf_{N\to\infty}E_N(0)\notag \\
&=E(0)
\end{align}
with $E(t)$ as in the statement of Theorem~\ref{thm:exist}. The identification of $\lim_{N\to\infty}E_N(0)=E(0)$ follows from from the convergence in \ref{2.4} along with \eqref{energy-6} with $t=0$. This establishes the energy inequality \eqref{energy-1stid}.

The identity \eqref{energy-2ndid} follows immediately from \eqref{energy-1stid} and the fundamental theorem after demonstrating that 
\begin{subequations}\label{energy-lims}\begin{align}
		\lim_{N\to\infty}\int_0^t \int_\O f(u_N(\tau))u_N'(\tau)\,dxd\tau &= \int_0^t \int_\O f(u(\tau))u'(\tau)\,dxd\tau,\label{energy-9a}\\
		\lim_{N\to\infty} \int_0^t \int_\G h(\t u_N(\tau))\t u_N'(\tau)\,dSd\tau &=\int_0^t \int_\G h(\t u(\tau))\t u'(\tau)\,dSd\tau.\label{energy-9b}
	\end{align}\end{subequations}
	The limits in \eqref{energy-lims} are readily verified from the convergence in Corollary~\ref{cor:converg}, and are omitted here.
\end{proof}

\section{Solutions for locally Lipschitz sources}\label{S3}
Having established the existence of solutions in the case where $f,h$ are globally Lipschitz we shall now relax the assumptions on these source functions.  This is accomplished by a truncation applied to the source terms in an approach similar to \cite{CEL1, GR, PRT-p-Laplacain, RW}.

One cannot expect the solutions obtained for locally Lipschitz sources to necessarily be global in time as was obtained in the preceding section for globally Lipschitz sources.  Moreover, to complete the proof of existence for more general sources in the following section it will be essential to track the dependencies of this finite existence time.   The necessary results are stated precisely in the following proposition:

\begin{prop}\label{prop:ll} Assume that the functions $f:\W\to\L[2]$ and $h\circ\t:\W\to\Lb[2]$ are locally Lipschitz continuous with constants $L_{f,K}$ and $L_{h,K}$ on the ball of radius $K$ about zero in $\W$. Then, problem \eqref{wave} possesses a local solution in the sense of Definition~\ref{def:weaksln} on an interval $[0,T_0]$.

	Further, the interval of existence for this solution depends only on the local Lipschitz constants $L_{f,K}$ and $L_{h,K}$ of $f:\W\to\L[6/5]$ and $h\circ\t:\W\to\Lb[4/3]$ on a ball of radius $K$ to be prescribed as a function of the initial energy $\E(0)$.  
\end{prop}
\begin{proof}
	Fix an arbitrary value of $K>p\E(0)$ and define the truncated source functions 
\begin{subequations}\begin{align}
f_K(u)&=\begin{cases}f(u)&\text{ for } ||u||_{1,p}\leq K,\\ f\left(\frac{Ku}{||u||_{1,p}}\right)&\text{ for }||u||_{1,p}>K,\end{cases}\label{ll-truncf}\\
h_K(\t u)&=\begin{cases}h(\t u)&\text{ for } ||u||_{1,p}\leq K,\\ h\left(\frac{K\t u}{||u||_{1,p}}\right)&\text{ for }||u||_{1,p}>K.\end{cases}\label{ll-trunch}
\end{align}\end{subequations}
It is readily verified that each $f_K:\W\to\L[2]$ and $h_K\circ\t:\W\to\Lb[2]$ is globally Lipschitz continuous, a proof of which can be found in \cite{PRT-p-Laplacain} and \cite{MOHNICK1}, for instance.  

Using these globally Lipschitz truncations, we may find by the results of Section~\ref{S2} a global solution $u$ to the corresponding $K$ problem
\begin{align}
\label{ll-wave}
\begin{cases}
u_{tt}-\Delta_p u -\Delta u_t = f_K(u) &\text{ in } \Omega \times (0,\infty),\\[.1in]
(u(0),u_t(0))=(u_0,u_1),\\[.1in]
|\grad u|^{p-2}\partial_\nu u + |u|^{p-2}u + \partial_\nu u_t + u_t = h_K(u)&\text{ on }\Gamma \times(0,\infty).
\end{cases}\tag{\thesection.K}
\end{align}

We shall now seek to find a sufficiently small interval $[0,T_0]$ upon which $||u(t)||_{1,p}\leq K$ upon which $f_K=f$ and $h_K=h$. That is, $u$ will be a solution to the non-truncated problem \eqref{wave} on $[0,T_0]$.  For this, several estimates on the source terms are first necessary.  

We start by choosing an arbitrary initial value of $T_1>0$, and note that from \eqref{apriori-a} we have $||u(t)||_{1,p}\leq M$ a.e. $[0,T_1]$ for some constant $M$.  Using the same type of estimate as in \eqref{apriori-3.4} we find that 
\begin{align}\label{ll-flip}
||f(u)||_{6/5}&\leq ||f(u)-f(0)||_{6/5}+||f(0)||_{6/5}\notag\\
&\leq L_{f,M}||u||_{1,p} + ||f(0)||_{6/5}\notag\\
&\leq C_M(||u||_{1,p}+1),
\end{align}
whereby mirroring the calculation in \eqref{apriori-3.5} we see that \begin{align}\label{ll-fbound}
\int_\O f(u)u'\,dx &\leq C_M(||u||_{1,p}^p+1)+\frac{1}{4}||u'||_{1,2}^2\quad\text{a.e. }[0,T_1]
\end{align}
with the constant $C_M$ derived from the local Lipschitz constant $L_{f,M}$. Analogously,
\begin{align}\label{ll-hbound}
\int_\G h(\t u)\t u'\,dS\leq C_M(||u||_{1,p}^p+1)+\frac{1}{4}||u'||_{1,2}^2\quad\text{a.e. }[0,T_1].
\end{align}
In particular, \eqref{ll-fbound} and \eqref{ll-hbound} imply that 
\begin{multline*}
\int_0^t \int_\O f_K(u)u'\,dxd\tau + \int_0^t \int_\G h_K(\t u)\t u'\,dSd\tau\\ \leq C_M \int_0^t (p\E(\tau) +1)\,d\tau + \int_0^t ||u'(\tau)||_{1,2}^2\,d\tau
\end{multline*}
on $[0,T_1]$, so that from the energy inequality \eqref{energy-2ndid} we see that the positive energy associated with a solution $u$ must satisfy  
\begin{align*}
\E(t) \leq \E(0) + C_M\int_0^t (p\E(\tau)+1)\,d\tau,\quad\t\in[0,T_1].  
\end{align*}
From Gronwall's inequality we thus obtain the bound \begin{align}\label{ll-ebound}\E(t)\leq (\E(0)+C_M t)\exp(pC_M t)\end{align}
on $[0,T_1]$. 
By choosing
\[
T_0 = \min \left\{T_1,\frac{K-p\E(0)}{pC_M},\frac{(p-1)\ln K}{pC_M}  \right\}
\]
we then find from \eqref{ll-ebound} that on the subinterval  $[0,T_0]\subset [0,T_1]$  we have  \begin{align*}
||u||_{1,p}^p \leq p\E(t) \leq p\underbrace{(\E(0)+C_Mt)}_\text{(i)}\underbrace{\exp(pC_Mt)}_\text{(ii)}\leq K^p
\end{align*}
given that $\text{(i)}\leq K/p$ and $\text{(ii)}\leq K^{p-1}$. This shows that $u$ is indeed a solution to the non-truncated problem on $[0,T_0]$; and since $||u(t)||_{1,p}\leq K\leq M$ on $[0,T_0]$ the length of the interval $[0,T_0]$ indeed depends only on the local Lipschitz constants $L_{f,K}$ and $L_{h,K}$, as desired.
\end{proof}

\smallskip

\section{General sources}\label{S4}  In order to establish the existence of solutions for more general sources we employ another truncation argument as in \cite{RW,PRT-p-Laplacain}. To begin, select as in \cite{Radu1} a sequence $\{\eta_n\}\subset C^\infty(\R)$ of cutoff functions such that 
\begin{align*}
0\leq\eta_n\leq 1,\quad |\eta_n'(s)|\leq \frac{C}{n},\quad\text{and }
\begin{cases}
\eta_n(s)=1,&\text{ for }|s|\leq n,\\
\eta_n(s)=0,&\text{ for }|s|>2n
\end{cases}
\end{align*}
for some constant $C$ independent from $n$ and define 
\begin{subequations}\begin{align}
f_n(u)&=f(u)\eta_n(u),\label{sc-f}\\ h_n(\t u) &= h(\t u)\eta_n(\t u)\label{sc-h}.
\end{align}\end{subequations}

 With these truncated sources we intend to build a sequence $\{u_n\}$ of approximate solutions where each $u_n$ satisfies the corresponding $n$-problem
\begin{align}\label{sc-n}
\begin{cases}
  u_{tt}-\Delta_p u -\Delta u_t = f_n(u) &\text{ in } \Omega \times (0,T),\\[.1in]
  (u(0),u_t(0))=(u_0,u_1),\\[.1in]
  |\grad u|^{p-2}\partial_\nu u + |u|^{p-2}u + \partial_\nu u_t + u_t = h_n(u)&\text{ on }\Gamma \times(0,T).
  \end{cases}\tag{\thesection.n}
  \end{align}

To do this we shall leverage the results of Section~\ref{S3} by showing that both $f_n$ and $h_n\circ\t$ are indeed locally Lipschitz as maps from $\W$ into $\L[2]$ and $\Lb[2]$, respectively.  In order to maintain a positive interval of existence for all of these approximate solutions we shall additionally need to find bounds on the local Lipschitz constants of these functions as maps into $\L[6/5]$ and $\Lb[4/3]$ which are, in an appropriate sense, independent of $n$.  In fact, these truncations satisfy even slightly more than these requirements.  The proof of this lemma is a routine series of estimates as in \cite[Sec. A]{MOHNICK1}, for instance.
\begin{lem}\label{lem:sc-lip}Each $f_n$ and $h_n$ given by \eqref{sc-f} and \eqref{sc-h} satisfy:
	\begin{enumerate}[(i)]
		\item $f_n:\W\to\L[2]$ and $h_n\circ\t:\W\to\Lb[2]$ are both globally Lipschitz continuous;
		\item $f_n:\W[1-\epsilon,p]\to\L[6/5]$ and $h_n\circ\t:\W[1-\epsilon,p]\to\Lb[4/3]$ are both locally Lipschitz continuous, and on any ball of radius $K$ these constants are \emph{independent} of $n$. That is, given any $K>0$ there exists a constant $C_K$ independent of $n$ such that 
		\begin{align*}
			||f_n(u)-f_n(v)||_{6/5},\,|h_n(\t u) - h_n(\t v)|_{4/3}\leq C_K ||u-v||_{1-\epsilon,p}
		\end{align*}
		for all $n$ and all $u,v\in\W$ with $||u||_{1-\epsilon,p}, ||v||_{1-\epsilon,p}\leq K$.
	\end{enumerate}
\end{lem}

With this truncation, each problem \eqref{sc-n} now possesses a solution $u_n$ in the sense of Definition~\ref{def:weaksln} from Proposition~\ref{prop:ll} in Section~\ref{S3} on an interval $[0,T]$.  It is important to note that this interval is indeed independent of $n$ precisely because the Lipschitz constants of $f_n$ and $h_n\circ\t$ as maps into $\L[6/5]$ and $\Lb[4/3]$, respectively, are independent of $n$ from Lemma~\ref{lem:sc-lip}.   Further, each $u_n$ satisfies the energy inequality 
\begin{align}\label{sc-energy}
\E_n (t) &+ \int_0^t ||u_n '(\tau)||_{1,2}^2\,d\tau \notag \\ & \leq 
 \E_n(0) + \int_0^t \int_\O f_n(u_n)u_n'\,dxd\tau + \int_0^t \int_\G h_n(\t u_n)\t u_n'\,dSd\tau
\end{align}
on $[0,T]$ where 
\begin{align*}
\E_n(t)=\frac{1}{2}||u_n'(t)||_2^2 + \frac{1}{p}||u_n(t)||_{1,p}^p.
\end{align*}

From this point the process of obtaining a solution to \eqref{wave} closely mirrors the process used in Section~\ref{S2}. The strategy is to first demonstrate that an analogue of the a priori estimates obtained in Proposition~\ref{prop:apriori} holds for the sequence $\{u_n\}$.  From this, we may identify a subsequence of these solutions along with a limit function $u$ which should satisfy \eqref{wave}.  As in Section~\ref{S2}, we shall have to pay particular attention to the convergence of the nonlinear terms; most especially those resulting from the $p$-Laplacian.  Many of these arguments are extremely similar to their counterparts earlier in the manuscript.  These parallel arguments have been cross-referenced, and as a result some of the proofs are intentionally terse where they are largely repetitive.

There is one major difference between the following proofs and those in Section~\ref{S2}.  Here, each approximate solution $u_n$ need not be in the form of a finite sum of separable functions and as such one cannot, for instance, obtain energy identities and a priori estimates via multiplication at the Galerkin level as was done in Proposition~\ref{prop:apriori}.  However, this distinction ends up not being of relatively minor significance since the definition of weak solution permits the use of time-dependent test functions.  In particular, $u_n$ is a valid test function which mimics the action of multiplying by $u_{N,j}$ at the Galerkin level.

\begin{prop}[c.f. Proposition~\ref{prop:apriori}]\label{prop:sc-bdds} The sequence $\{u_n\}$ of solutions to the $n$-problem \eqref{sc-n} satisfies 
	\begin{subequations}\begin{alignat}{3}
		\{u_n\}_1^\infty&\text{ is a bounded sequence in }&&L^{\infty} (0,T;\W), \label{scbdd-a}\\
		\{u_n'\}_1^\infty&\text{ is a bounded sequence in }&&L^{\infty} (0,T;\L[2]),\label{scbdd-b}\\
		\{u_n'\}_1^\infty&\text{ is a bounded sequence in }&&L^2(0,T;\H[1]),\label{scbdd-c}\\
		\{u_n''\}_1^\infty&\text{ is a bounded sequence in }&&L^2(0,T;(\W)').\label{scbdd-d}
		\end{alignat}\end{subequations}
\end{prop}
\begin{proof}
	Notice first that \eqref{scbdd-a} and \eqref{scbdd-b} are, in fact, immediate from Proposition~\ref{prop:ll} which reveals that $\E_n$ is bounded uniformly in $n$ almost everywhere on $[0,T]$.
	
	Using the estimates \eqref{ll-fbound} and \eqref{ll-hbound} established in the previous section we find that

	\begin{align}\label{sc-srcbounds}
	\int_\O f_n(u_n)u_n'\,dx + \int_\G h_n(\t u_n)\t u_n'\,dS \leq C_K(||u_n||_{1,p}^p+1) + \frac{1}{2}||u_n'||_{1,2}^2.	
	\end{align}
	The constant $C_K$ in \eqref{sc-srcbounds} is indeed independent of $n$ in light of Lemma~\ref{lem:sc-lip}, as the value of $K$ is derived only from the initial energy $\E(0)$ which is independent of $n$.
	
	Apply Equation~\eqref{sc-srcbounds} to \eqref{sc-energy} it is seen that each $u_n$ satisfies 
	\begin{align*}
	\E_n(t) + \frac{1}{2} \int_0^t ||u_n'(\tau)||_{1,2}^2\,d\tau \leq \E_n(0) + C\int_0^t(1+\E_n(\tau))\,d\tau; \quad t\in[0,T].
	\end{align*} The desired conclusion of \eqref{scbdd-c} follows from Gronwall's inequality exactly as in Proposition~\ref{prop:apriori}.
	
	For \eqref{scbdd-d}, since each $u_n$ satisfies \eqref{slnid} we have that for each $\phi\in\W$ and $t\in[0,T]$ that 
	\begin{multline*}
	|\langle u_n''(t),\phi\rangle_p| = \left| \frac{d}{dt}(u_n'(t),\phi)_\O \right| \\
	\leq 
	\underbrace{|\langle -\Delta_p u_n(t),\phi\rangle_p| + |\langle -\Delta_2 u_n(t),\phi\rangle_2 |}_\text{(i)} + \underbrace{|(f_n(u_n(t)),\phi)_\O|}_\text{(ii)} + \underbrace{|(h_n(\t u_n(t)),\t \phi)_\G}_\text{(iii)}.
	\end{multline*}
	Using the bounds on $-\Delta_p$ from (\ref{plaplace-opnorm}), 
	\begin{align*}
	\text{(i)}\leq 2(||u_n(t)||_{1,p}^{p-1} + ||u_n'(t)||_{1,2})||\phi||_{1,p}. 
	\end{align*}
	From H\"older's inequality, 
	\begin{align*}
	 \text{(ii)}&\leq ||f_n(u_n(t))||_{6/5}||\phi||_6\\
	 &\leq C(||u_n(t)||_{1,p}+1)||\phi||_{1,p}
	\end{align*}
	from the same arguments as in \eqref{ll-flip}.   Similarly,
	\begin{align*}
	\text{(iii)}\leq C(||u_N(t)||_{1,p}+1)||\phi||_{1,p}.
	\end{align*}
	 Having demonstrated that 
	\begin{align*}
	|\langle u_n ''(t),\phi\rangle_p| \leq C\Big(||u_n(t)||_{1,p}^{p-1}+||u_n(t)||_{1,p}+||u_n'(t)||_{1,2}+1\Big)
	||\phi||_{1,p}
	\end{align*}
	the proof of \eqref{scbdd-d} thus follows immediately from \eqref{scbdd-a} and \eqref{scbdd-c}.
\end{proof}

As was the case in Corollary~\ref{cor:converg}, the standard compactness theorems yield the following:  

\begin{cor}[c.f. Corollary~\ref{cor:converg}]\label{cor:scconverg} For all sufficiently small $\epsilon>0$ there exists a function $u$ and a subsequence of $\{u_n\}$ (still denoted $\{u_n\}$) such that  
	\begin{subequations} \begin{alignat}{2}
		u_n&\to u &\text{ weak* in }&L^\infty(0,T;\W), \label{scconverg:a} \\
		u_n'&\to u' &\text{ weak* in }&L^\infty(0,T;\L[2]),\label{scconverg:b} \\
		u_n'&\to u' &\text{ weakly in }&L^2(0,T;\H[1]),\label{scconverg:c} \\
		u_n &\to u &\text{ strongly in }&C([0,T];\W[1-\epsilon,p]),\label{scconverg:d} \\
		u_n' &\to u' &\text{ strongly in }&L^2(0,T;\W[1-\epsilon,2])\label{scconverg:e}\\
		u_n''&\to u''&\text{ weakly in }&L^2(0,T;(\W)').\label{scconverg:f}
		\end{alignat} \end{subequations} 
\end{cor}

The following three results follow immediately and are given here without proofs as they remain entirely unchanged from the corresponding results in Section~\ref{S2}.

\begin{cor}[c.f. Corollary~\ref{cor:aeinw1p}]On a subsequence,
	\begin{align*}
	u_n(t)\to u(t)\text{ weakly in }\W\text{ for a.e. }t\in[0,T].
	\end{align*}
\end{cor}

\begin{cor}[c.f. Corollary~\ref{cor:uinCw}]The limit function $u$ identified in Corollary~\ref{cor:scconverg} satisfies $u\in C_w([0,T];\W)$ and $u'\in C_w([0,T];\L[2])$.
	\end{cor}

\begin{prop}[c.f. Proposition~\ref{prop:ut}]\label{prop:sc-ut} With $\{u_n\}$ and $u$ as in Corollary~\ref{cor:scconverg},
	$$-\Delta_2 u_n' \to -\Delta_2 u'\text{ weakly in }L^2(0,T;(\W[1,2])').$$
\end{prop}

We have thus produced a sequence $\{u_n\}$ of solutions each satisfying the identity 
\begin{multline}\label{sc-slnid}
	(u_n'(t),\phi(t))_\Omega - (u_1,\phi(0))_\Omega -\int_0^t(u_n'(\tau),\phi_t(\tau))_\Omega\,d\tau \\
	+ \int_0^t \langle -\Delta_p u_n(\tau), \phi(\tau)\rangle_p\,d\tau  +\int_0^t \langle -\Delta u_n'(\tau), \phi(\tau)\rangle_2 \,d\tau\\ 
	=\int_0^t\int_\Omega f_n(u_n(\tau))\phi(\tau)\,dxd\tau 
	+ \int_0^t \int_\Gamma h_n(\t u_n(\tau))\t\phi(\tau)\,dSd\tau 
\end{multline} 
for all test functions $\phi\in C_w([0,T];\W)$ with $\phi_t\in L^2(0,T;\W[1,2])$ in lieu of a fixed sources $f$ and $h$. In order to pass to the limit in \eqref{sc-slnid} we shall first establish that the source terms converge in an appropriate sense.

\begin{prop}[c.f. Proposition~\ref{prop:limf}]\label{prop:sclimf}
\begin{subequations}\begin{alignat}{3}
	f_n(u_n)&\to f(u)&\text{ strongly in }&L^\infty(0,T;\L[6/5]),\label{prop:sclimf-a}\\
	h_n(\t u_n)&\to h(\t u)&\text{ strongly in }&L^\infty(0,T;\Lb[4/3]).\label{prop:sclimf-b}
\end{alignat}\end{subequations}
\end{prop}

\begin{proof}
	From the triangle inequality,
	\begin{align*}
	||f_n(u_n(t))-f(u(t))||_{6/5}\leq \underbrace{||f_n(u_n(t))-f_n(u(t))||_{6/5}}_\text{(i)}+\underbrace{||f_n(u(t))-f(u(t))||_{6/5}}_\text{(ii)}.	
	\end{align*}
	Since each $f_n:\W[1-\epsilon,p]\to\L[6/5]$ is locally Lipschitz with a constant independent of $n$ from Lemma~\ref{lem:sc-lip}, we obtain from \eqref{scconverg:d} that 
	\begin{align*}
	\text{(i)} \leq C||u_n(t)-u(t)||_{1-\epsilon,p}\to 0\text{ for a.e. }t\in[0,T]
	\end{align*}
	exactly as was the case in Proposition~\ref{prop:limf}.
	
	For $\text{(ii)}$, it is clear that $f_n(u(t))\to u(t)$ pointwise a.e. on $\O$. Using the pointwise bound 
	$$|f_n(u(t))-f(u(t))|= |\eta_n(u(t))-1||f(u(t))|\leq |f(u(t))|$$
	which is $\L[6/5]$ given that $||f(u(t))||_{6/5}\leq C(||u(t)||_{1-\epsilon,p}+1)$ by Lemma~\ref{lem:f-lipschitz}, we obtain 
	\begin{align*}
	||f_n(u(t))-f(u(t))||_{6/5} \to 0 \text{ a.e. }t\in[0,T]
	\end{align*}
	by the Lebesgue dominated convergence theorem.
	
	For $h_n$ the proof is identical as once again we find 
	\begin{align*}
	|h_n(\t u_n) - h(\t u)|_{4/3} \leq \underbrace{|h_n(\t u_n) - h_n(\t u)|_{4/3}}_\text{(i)}+ \underbrace{|h_n(\t u)-h(\t u)|_{4/3}}_\text{(ii)}	
	\end{align*}
	with $\text{(i)}\to 0$ by Lemma~\ref{lem:sc-lip} along with \eqref{scconverg:d} and $\text{(ii)}\to 0$ by the dominated convergence theorem.
\end{proof}

Since each approximate solution $u_n$ satisfies 
\begin{align}\label{sc-ICs}
(u_n(0),u_n'(0))=(u_0,u_1)\text{ in }\W\times\L[2]
\end{align}
it is clear that an analogue of (\ref{2.4}) is not required here.  We thus turn our attention to the much more difficult task of verifying the convergence of the terms due to the $p$-Laplacian.

\begin{prop}[c.f. Proposition~\ref{prop:limlaplace}]\label{prop:sclimlaplace}On a subsequence, the approximate solutions $\{u_n\}$ to the $n$-problem \eqref{sc-n} along with the limit function $u$ identified in Corollary~\ref{cor:scconverg} satisfy 
	\begin{align}
	-\Delta_p u_n \to -\Delta_p u\text{ weak* in }L^\infty(0,T;(\W)').	
	\end{align}
\end{prop}
\begin{proof} We shall inherit the framework of the proof of Proposition~\ref{prop:limlaplace} by taking $X=L^p(0,T;\W)$ and the $p$-Laplacian extended to a maximal monotone operator $-\Delta_p:X\to X'$.  As in that proof, the bounds from \eqref{scconverg:a} along with the operator norm bound on $-\Delta_p$ in (\ref{plaplace-opnorm}) again permit the conclusion that $\{-\Delta_p u_n\}$ is a bounded sequence in $L^\infty(0,T;(\W)')$ so that 
	\[
	-\Delta_p u_n \to \eta\text{ weak* in }L^\infty(0,T;(\W)')
	\]
	for some $\eta$.  Likewise, to conclude that $\eta = -\Delta_p u$ it again enough to show instead that 
	\begin{align}\label{sc-limlaplacewts}
	\limsup_{n\to\infty}\langle -\Delta_p u_n, u_n\rangle_{X',X} \leq \langle \eta,u\rangle_{X',X}.
	\end{align}
	
	By taking $\phi=u_n$ in \eqref{sc-slnid} and rearranging we obtain 
	\begin{multline}\label{sclaplace1}
	\int_0^t \langle -\Delta_p u_n(\tau),u_n(\tau)\rangle_p \,d\tau = -(u_n'(t),u_n(t))_\O + (u_1,u_0)_\O \\
	+ \int_0^t ||u_n'(\tau)||_2^2\,d\tau 
	- \int_0^t \langle -\Delta u_n'(\tau),u_n(\tau)\rangle_2\,d\tau \\ 
	+ \int_0^t \int_\O f_n(u_n(\tau))u_n(\tau)\,dxd\tau 
	+ \int_0^t \int_\G h_n(\t u_n(\tau))\t u_n(\tau)\,dSd\tau.
	\end{multline}
	However, from the product rule in the distributional sense we have both
	\begin{align}
	\frac{d}{dt}(\grad u_n,\grad u_n)_\O &=  2(\grad u_n',\grad u_n)_\O,\\
	\frac{d}{dt}(\t u_n,\t u_n)_\G &=2(\t u_n',\t u_n)_\G
	\end{align}
	since $\grad u_n\in W^{1,2}(0,T;L^2(\O))$ from \eqref{scconverg:a} and \eqref{scconverg:c}. Thus, 
	\begin{align*}
	\int_0^t \langle -\Delta u_n'(\tau),u_n(\tau)\rangle_2 \,d\tau & = \int_0^t\int_\O \grad u_n'(\tau)\cdot \grad u_n(\tau)\,dxd\tau + \int_0^t \int_\G \t u_n'(\tau)\t u_n(\tau)\,dSd\tau \notag\\
	&=\frac{1}{2}\int_0^t\int_\O \frac{d}{dt} |\grad u_n(\tau)|^2\,dxd\tau 
	+ \int_0^t \int_\G\frac{d}{dt}|\t u_n'(\tau)|^2\,dSd\tau\notag \\
	& = \frac{1}{2}||u_n(t)||_{1,2}^2 - \frac{1}{2}||u_n(0)||_{1,2}^2
	\end{align*}
	so that we may rewrite \eqref{sclaplace1} as 
	\begin{multline}\label{sclaplace2}
	\int_0^t \langle -\Delta_p u_n(\tau),u_n(\tau)\rangle_p \,d\tau = \underbrace{-(u_n'(t),u_n(t))_\O}_\text{(i)} + (u_1,u_0)_\O\\
	+ \underbrace{\int_0^t ||u_n'(\tau)||_2^2\,d\tau}_\text{(ii)} 
	+\underbrace{\frac{1}{2}||u_n(0)||_{1,2}^2}_\text{(iii)}
	-\underbrace{\frac{1}{2}||u_n(t)||_{1,2}^2}_\text{(iv)}\\
	+ \underbrace{\int_0^t \int_\O f_n(u_n(\tau))u_n(\tau)\,dxd\tau 
	+ \int_0^t \int_\G h_n(\t u_n(\tau))\t u_n(\tau)\,dSd\tau}_\text{(v)}.\ptag{sclaplace1}
	\end{multline}
	As this expression is equivalent to \eqref{limlaplace-s1pp} in the proof of the corresponding Proposition~\ref{prop:limlaplace} and as the sequence $\{u_n\}$ enjoys the same convergence properties as in that proof, we are justified in taking the limit superior in each of the terms $\text{(i)}$ through $\text{(v)}$ of \eqref{sclaplace2} to obtain 
	\begin{multline}\label{sclaplace3}
	\limsup_{n\to\infty} \int_0^t \langle -\Delta_p u_n,u_n\rangle_p\,d\tau \leq (u'(0),u(0))_\O - (u'(t),u(t))_\O + \int_0^t ||u'(\tau)||_2^2\,d\tau\\ 
	+\frac{1}{2}||u(0)||_{1,2}^2 - \frac{1}{2}||u(t)||_{1,2}^2\\
	+ \int_0^t\int_\O f(u)u\,dxd\tau + \int_0^t \int_\G h(\t u)\t u\,dSd\tau\quad\text{a.e. }[0,T].
	\end{multline}
	Again we seek to express the right hand side of \eqref{sclaplace3} by effecting a limit through a different means.  By taking $\phi=u$ in \eqref{sc-slnid} we obtain
	\begin{multline}\label{limlaplace4}
	\int_0^t \langle -\Delta_p u_n(\tau), u(\tau)\rangle_p\,d\tau = 
	-\underbrace{(u_n'(t),u(t))_\Omega}_\text{(i)} + (u_1,u_0)_\Omega\\ +\underbrace{\int_0^t(u_n'(\tau),u'(\tau))_\Omega\,d\tau}_\text{(ii)} 
	-\underbrace{\int_0^t \langle -\Delta u_n'(\tau), u(\tau)\rangle_2 \,d\tau}_\text{(iii)}\\ 
	+\underbrace{\int_0^t\int_\Omega f_n(u_n(\tau))u(\tau)\,dxd\tau 
	+ \int_0^t \int_\Gamma h_n(\t u_n(\tau))\t u (\tau)\,dSd\tau}_\text{(iv)} 
	\end{multline} 
	Taking the limit as $n\to\infty$ is easily justified in each of the preceding terms as it was in Proposition~\ref{prop:limlaplace}.  Given that 
	\begin{align*}
	\langle -\Delta u'(t),u(t)\rangle_2 = \frac{1}{2}\frac{d}{dt}||u(t)||_{1,2}^2 
	\end{align*}
	we find that \eqref{sc-limlaplacewts} indeed holds as was the case in the Proposition~\ref{prop:limlaplace}, completing the proof.	
\end{proof}

It is now clear that we may take the limit in \eqref{sc-slnid} to conclude that the limit function $u$ is a weak solution in the sense of Definition~\ref{def:weaksln}. It thus remains to show only that $u$ verifies the required energy inequalities.

\begin{prop}[c.f. Proposition~\ref{prop:energy}]\label{prop:sc-energy}\sloppy The limit function $u$ identified in Corollary~\ref{cor:scconverg} satisfies the energy inequalities \eqref{energy-2ndid} and \eqref{energy-1stid} in the statement of Theorem~\ref{thm:exist}.
\end{prop}

\begin{proof}
Since each $u_n$ verifies \eqref{energy-2ndid} we obtain 
\begin{multline*}\E_n(t)+\int_0^t ||u_n'(\tau)||_{1,2}^2\,d\tau \leq \E_n(0) \\
+ \int_0^t \int_\O f_n(u_n(\tau))u_n'(\tau)\,dxd\tau + \int_0^t \int_\G h_n(\t u_n(\tau))\t u_n'(\tau)\,dSd\tau\end{multline*}
with positive energy 
$$\E_n(t)=\frac{1}{2}||u_n(t)||_2^2+\frac{1}{p}||u_n(t)||_{1,p}^p.$$
From Proposition~\ref{prop:sclimf} we have $f_n(u_n)\to f(u)$ strongly in $L^2(0,T;\L[6/5])$ and from \eqref{scconverg:e} along with the embedding $\W[1-\epsilon,2]\into \L[6]$ for sufficiently small $\epsilon>0$ we have $u_n'\to u'$ strongly in $L^2(0,T;\L[6])$. Thus, 
$$\lim_{n\to\infty} \int_0^t \int_\O f_n(u_n(\tau))u_n'(\tau)\,dxd\tau=\int_0^t \int_\O f(u(\tau))u'(\tau)\,dxd\tau.$$
Similarly, Proposition~\ref{prop:sclimf} and \eqref{scconverg:c} along with the trace $\W[1,2]\tinto\Lb[4]$ yield  
$$\lim_{n\to\infty} \int_0^t \int_\G h_n(\t u_n(\tau))\t u_n'(\tau)\,dSd\tau=\int_0^t \int_\G h(\t u(\tau))\t u'(\tau)\,dSd\tau$$
by the usual ``weak-strong'' convergence result since $h_n(\t u_n) \to h(\t u)$ strongly in $L^2(0,T;\Lb[4/3])$ and $\t u_n' \to \t u'$ weakly in $L^2(0,T;\Lb[4])$.   Using weak lower semicontinuity we may thus establish \eqref{energy-2ndid} since 
\begin{align*}
\E(t) + \int_0^t ||u'(\tau)||_{1,2}^2\,d\tau 
&\leq \liminf_{n\to\infty}\left(\int_0^t ||u_n'(\tau)||_{1,2}^2\,d\tau +\E_n(t) \right) \\
&=\E(0)+\int_0^t \int_\O f(u(\tau))u'(\tau)\,dxd\tau \\
&\qquad\qquad + \int_0^t \int_\G h(\t u(\tau))\t u'(\tau)\,dSd\tau.
\end{align*}

In order to obtain the final identity \eqref{energy-1stid} we need only note that the absolutely continuous functions  $F(u)=\int_0^uf(s)\,ds$ and $H(\t u)=\int_0^{\t u}h(s)\,ds$ satisfy
\begin{gather*}
\frac{d}{dt}F(u(\tau))=f(u(\tau))u'(\tau)\quad\text{and}\quad\frac{d}{dt}H(\t u(\tau))=h(\t u(\tau))\t u'(\tau)
\end{gather*}
for a.e. $\tau\in[0,T]$, from which the result follows from \eqref{energy-2ndid} and the fundamental theorem of calculus. 
\end{proof}

This completes the proof of Theorem~\ref{thm:exist}.

\section{Global existence}\label{S5}
 It has been shown in Section~\ref{S2} that global solutions of \eqref{wave} exist in the case where $f$ and $h$ are both globally Lipschitz functions from $\W$ to $\L[2]$ and $\Lb[2]$, respectively.  In general this condition is only assured by taking $q=r=1$ which corresponds essentially to linear source terms.  

As in \cite{GR, PRT-p-Laplacain} it is the case here that either a given solution $u$ must, in fact, be global in time or else one may find a value of $T_0$ with $0<T_0<\infty$ so that 
\begin{align}\label{glob-1}
\limsup_{t\to T_0^-}\left( \E(t) + \int_0^t ||u'(\tau)||_{1,2}^2\,d\tau \right)=\infty
\end{align}
with positive energy $\E(t)=\frac{1}{2}||u'(t)||_2^2 + \frac{1}{p}||u(t)||_{1,p}^p$ from Theorem~\ref{thm:exist}. By demonstrating a bound on the energy 
\begin{align*}
\E(t) + \int_0^t ||u'(\tau)||_{1,2}^2\,d\tau
\end{align*}
on every interval $[0,T]$ which is dependent only upon $T$ and the positive initial  energy  $\E(0)$, we shall show that the situation in \eqref{glob-1} cannot occur as the argument is bounded on any finite interval.  This bound is only possible provided the exponents of the source terms are sufficiently small, specifically when $q,r\leq p/2$.  The following proposition thus establishes the desired result.

\begin{prop}If $u$ is a weak solution of \eqref{wave} given by Theorem~\ref{thm:exist} on $[0,T]$ and $r,q\leq p/2$, then there exists a constant $M$ dependent upon $T$ and $\E(0)$ so that 
	\begin{align*}
	\E(t) + \int_0^t ||u'(\tau)||_{1,2}^2\,d\tau<M,\qquad t\in[0,T].
	\end{align*}
\end{prop}
\begin{proof}As $u$ satisfies the energy inequality 
	\begin{multline}\label{glob-energyid}
	\E(t) + \int_0^t ||u'(\tau)||_{1,2}^2\,d\tau \leq \E(0) \\
	+ \int_0^t \int_\O f(u(\tau))u'(\tau)\,dxd\tau + 
	\int_0^t \int_\G h(\t u(\tau))\t u'(\tau)\,dSd\tau,
	\end{multline}
	from Theorem~\ref{thm:exist} we may bound the desired quantity using Gronwall's inequality provided the source terms in this expression can be adequately controlled.  From the pointwise bound $|f(u)|\leq C(|u|^q+1)$ in Remark~\ref{rmk:fbound} it follows that 
	\begin{align*}
	||f(u(\tau))||_2^2&\leq C\int_\O(1+|u(\tau)|^q)^2 \,dx
	\leq C(1+||u(\tau)||_{2q}^{2q}).  
	\end{align*}
	Using H\"older's inequality followed by Young's inequality with $\epsilon$ we may thus estimate that 
	\begin{align}\label{glob-e2}
	\int_\O f(u(\tau))u'(\tau)\,dx
	&\leq ||f(u(\tau))||_2||u'(\tau)||_2\notag \\
	&\leq C(1+||u(\tau)||_{2q}^{2q}) + \frac{1}{4}||u'(\tau)||_{1,2}^2	
	\end{align}
	for a suitable choice of $\epsilon$ relative to the constant associated with the embedding $\W[1,2]\into\L[2]$. Under the assumption that $2q\leq p$ it follows that $\W\into\L[2q]$, so that $||u(\tau)||_{2q}^{2q} \leq C||u(\tau)||_{1,p}^{2q}\leq C(||u(\tau)||_{1,p}^p+1).$ Thus, from \eqref{glob-e2} we obtain
	\begin{align}\label{glob-e3}
	\int_\O f(u(\tau))u'(\tau)\,dx \leq C(||u(\tau)||_{1,p}^p+1) + \frac{1}{4}||u'(\tau)||_{1,2}^2.
	\end{align}
	
	The same argument applied to the source term $h$ utilizing the continuity of the trace operators $\W[1,2]\tinto\Lb[2]$ and $\W[1,p]\tinto\Lb[2r]$ yields 
	\begin{align}\label{glob-e4}
	\int_\G h(\t u(\tau))\t u'(\tau)\,d\tau \leq C(||u(\tau)||_{1,p}^p+1) + \frac{1}{4}||u'(\tau)||_{1,2}^2.
	\end{align}
	Integrating \eqref{glob-e3} and \eqref{glob-e4} on $[0,t]$ and recalling that $||u(\tau)||_{1,p}^p \leq p\E(\tau)$ we thus obtain from \eqref{glob-energyid} that 
	\begin{gather}\label{glob-e5}
	\E(t) + \frac{1}{2}\int_0^t ||u'(\tau)||_{1,2}^2\,d\tau \leq \E(0) + C\int_0^t (1+\E(t))\,d\tau.
	\end{gather} 
	From Gronwall's inequality, \eqref{glob-e5} implies that 
	$$\E(t)\leq (\E(0)+Ct)\exp(Ct)$$
	so that $\E(t)\leq N$ on any interval $[0,T]$ by taking $N=(\E(0)+CT)\exp(CT)$.  From \eqref{glob-e5}, we then obtain  
	\begin{gather*}
	\E(t)+\frac{1}{2}\int_0^t||u'(\tau)||_{1,2}^2\,d\tau \leq \E(0)+ CT(1+N)\text{ for }t\in[0,T]
	\end{gather*}
	and the desired result follows by selecting the constant $M=2\E(0)+2CT(1+N)$.
\end{proof}

\section{Blow-up}\label{S6}
The goal of this section is to demonstrate that solutions to \eqref{wave} necessarily exist only on a finite interval of time provided the source feedback terms $f$ and $h$ are of sufficient magnitude. Assume in line with Assumption~\ref{ass:blowupsrc} that  $f$ and $h$ are of the form
	\begin{alignat*}{6}
	f(s)&=(q+1)|s|^{q-1}s&\text{ with }&&p-1&<q&<\frac{5p}{2(3-p)},\\
	h(s)&=(r+1)|s|^{r-1}s&\text{ with }&&p-1&<r&<\frac{3p}{2(3-p)}.
	\end{alignat*}

\begin{rmk}\label{blowup-srcrmk}
	\sloppy In particular, notice that $f(s)=\frac{d}{ds}|s|^{q+1}$. As such, we may explicitly compute that 
	\begin{gather*}
	F(u)=\int_0^u f(s)\,ds = |u|^{q+1}\quad\text{ so that }\quad \int_\O F(u)\,dx = ||u||_{q+1}^{q+1}. 
	\end{gather*}
	The same calculations on $h$ yield 
	\begin{gather*}
	H(\t u)=\int_0^{\t u} h(s)\,ds = |\t u|^{r+1}\quad\text{ and }\quad \int_\G H(\t u)\,dS = |\t u|_{r+1}^{r+1}. 
	\end{gather*}
	As such, the energy inequalities \eqref{energy-2ndid} and \eqref{energy-1stid} are equivalent, and may be expressed as 
	\begin{align}\label{b-energyid}
	E(t)+\int_0^t ||u'(\tau)||_{1,2}^2\,d\tau \leq E(0)
	\end{align}
	with total energy 
	\begin{align}\label{b-ttlenergy}
	E(t)=\frac{1}{2}||u'(t)||_2^2 + \frac{1}{p}||u(t)||_{1,p}^p - ||u(t)||_{q+1}^{q+1} -|\t u(t)|_{r+1}^{r+1}.
	\end{align}
	As it occurs in the proof of Theorem~\ref{thm:blowup}, it is additionally useful to notice that 
	\begin{gather*}
	\int_\O f(u)u\,dx = \int_\O (q+1)|u|^{q-1}u^2 = (q+1)||u||_{q+1}^{q+1}\\
	\text{and, similarly}\\
	\quad\int_\G h(\t u)\t u\,dS = (r+1)|\t u|_{r+1}^{r+1}.
	\end{gather*}
\end{rmk}

We may now prove the main result of this section.
\begin{proof}[Proof of Theorem~\ref{thm:blowup}]
	Given any weak solution $u$ to \eqref{wave} we define the lifespan, $T$, of the solution to be the supremum over all $T'>0$ such that $u$ is a solution to \eqref{wave} on $[0,T']$ in the sense of Definition~\ref{def:weaksln}.  
	By establishing a lower bound on the growth of an appropriate functional we shall show that this value of $T$ must be finite, and additionally that $$\limsup_{t\to T^-}\E(t)=\infty.$$
	
	As in \cite{AR2, BL3, GR1, PRT-p-Laplacain} we introduce the functions 
	\begin{gather*}
	G(t)=\int_0^t ||u'(\tau)||_{1,2}^2\,d\tau -E(0),\qquad N(t)=||u(t)||_2^2,\\
	S(t)=||u(t)||_{q+1}^{q+1} + |\t u(t)|_{r+1}^{r+1}
	\end{gather*}
	for $t\in[0,T)$ with total energy $E(t)=\E(t)-S(t)$ just as in \eqref{b-ttlenergy}. 
	Since $u' \in L^2(0,T;\W[1,2])$ the function $G$ is seen to be absolutely continuous with $$G'(t)=||u'(t)||_{1,2}^2\geq 0\quad\text{ a.e. }[0,T]$$ and $G(0)=-E(0)>0$ by assumption, from which it follows that $G$ is a positive, increasing function on $[0,T]$.  Moreover, with this choice of functions the energy inequality \eqref{b-energyid} may be written succinctly as 
	\begin{align}\label{b-energy2}
	G(t)\leq -E(t) = S(t)-\E(t).
	\end{align}
	
	The differentiability of $N$ is an essential component of the remainder of the proof.  Writing $N(t)=(u(t),u(t))_\O$ and noticing that $u,u'\in L^2(0,T;\L[2])$ we may apply a product rule in the distributional sense (see, for instance, \cite[Prop 1.2]{Sh} with $V=V'=H=\L[2]$) so that 
	\begin{align}\label{blowup-np}
	N'(t)=\frac{d}{dt}(u(t),u(t))_\O = 2(u'(t),u(t))_\O.\end{align}
	By taking $\phi(t)=u(t)$ as a test function in the variational identity \eqref{slnid} we obtain 
	\begin{multline*}
	\overbrace{(u'(t),u(t))_\O}^{\frac{1}{2}N'(t)} =   (u'(0),u(0))_\O + \int_0^t ||u'(\tau)||_2^2\,d\tau \\
	-\int_0^t \langle -\Delta_p u(\tau),u(\tau)\rangle_p\,d\tau 
	-\int_0^t \langle -\Delta_2 u'(\tau),u(\tau)\rangle_2 \,d\tau\\
	+\int_0^t \int_\O f(u(\tau))u(\tau)\,dxd\tau 
	+ \int_0^t \int_\G h(\t u(\tau))\t u(\tau)\,dSd\tau  
	\end{multline*}
	As $N'$ is now seen to be absolutely continuous, we may differentiate again to conclude that 
	\begin{multline}\label{blowup-npp}
	\frac{1}{2}N''(t)=||u'(t)||_2^2 
	- \overbrace{\langle \Delta_p u(t),u(t)\rangle_p}^\text{(i)} 
	- \langle -\Delta_2 u'(t),u(t)\rangle_2 \\
	+ \underbrace{\int_\O f(u(t))u(t)\,dx 
	+ \int_\G h(\t u(t))\t u(t)\,dS}_\text{(ii)}.
	\end{multline}
	
	By definition, we may express (i) in \eqref{blowup-npp} as 
		\begin{align*}
		\text{(i)}=\int_\O |\grad u(t)|^{p-2}\grad u(t) \cdot \grad u(t) \,dx + \int_\G |\t u(t)|^{p-2}\t u(t) \t u(t)\,dS=||u(t)||_{1,p}^p,
		\end{align*}
	and from Remark~\ref{blowup-srcrmk} we may express (ii) in \eqref{blowup-npp} as  
	\begin{align*}
	\text{(ii)} &= (q+1)||u||_{q+1}^{q+1} + (r+1)|\t u|_{r+1}^{r+1}.	\end{align*}
	Thus, we may express \eqref{blowup-npp} equivalently as 
	\begin{multline}\label{blowup-npp2}
	N''(t)= 2||u'(t)||_2^2 - 2||u(t)||_{1,p}^p 
		- 2\langle -\Delta_2 u'(t),u(t)\rangle_2\\
		+2(q+1)||u(t)||_{q+1}^{q+1} + 2(r+1)|\t u(t)|_{r+1}^{r+1}.
					\end{multline}	
					
	As it will be used throughout the remainder of the proof it is useful to pause and notice that for real numbers $0<\eta<1$ and $\delta,z\geq 0$ we have  $$z^\eta \leq z+1 \leq z+1+\delta+\frac{z}{\delta} = \left(1+\frac{1}{\delta}\right)(\delta + z).$$   By taking $\delta=G(0)>0$ and using the fact that $G$ is an increasing function, it then follows that 
	\begin{align}\label{b-fracexp}
	z^\eta \leq C(G(0)+z) \leq C(G(t)+z)	
	\end{align}for the constant $C=(1+1/G(0))$.					
	
	With these preliminaries established, we now define the function 
	$$Y(t)=G(t)^{1-\alpha} + \beta N'(t)$$
	for constants $0<\alpha,\beta<1/2$ to be determined  later.  Our ultimate goal shall be to demonstrate that 
	\begin{align}\label{b-wts}
	Y'(t)\geq CY(t)^{1/(1-\alpha)}%
	\end{align}
	with $Y(0)>0$ from which the desired result will follow given that $1<1/(1-\alpha)<2$.  This is accomplished in two steps beginning first with the right hand side of \eqref{b-wts}.

	{\bf Step 1:} We show here that $Y(t)^{1/1-\alpha} \leq C_1\left[G(t)+ ||u'(t)||_2^2 + ||u(t)||_{1,p}^p \right]$ for a constant $C_1>0$.\\
	Beginning with the definition of $Y$ and $N'$ we find that 
	\begin{align}\label{b-c1-1}
	Y(t)&= \left[ G(t)^{1-\alpha} + 2\beta(u'(t),u(t))_\O\right]^{1/(1-\alpha)}\notag \\
	&\leq C\left[ G(t) + ||u'(t)||_{2}^{\theta/(1-\alpha)} + ||u(t)||_2^{\theta'/(1-\alpha)}  \right]\notag \\
	&= C\left[ G(t) + ||u'(t)||_{2}^2 + ||u(t)||_2^{\theta'/(1-\alpha)}  \right]
	\end{align}
	from H\"older's inequality followed by Young's inequality with conjugate exponents $\theta=2(1-\alpha)>1$ and $\theta'=2(1-\alpha)/(1-2\alpha)$.  Since 
	$$\frac{1}{p}\frac{\theta'}{1-\alpha} = \frac{2}{p(1-2\alpha)}\to \frac{2}{p}<1\text{ as }\alpha\to 0^+$$
	we may select $\alpha$ sufficiently small so that $\theta'/p(1-\alpha)<1$, whereby
	\begin{align}\label{b-c1-2}
	||u(t)||_2^{\theta'/(1-\alpha)} & \leq C||u(t)||_{1,p}^{\theta'/(1-\alpha)}= C(||u(t)||_{1,p}^p )^{\theta'/p(1-\alpha)}\notag\\
	&\leq C(G(t)+||u(t)||_{1,p}^p)
	\end{align}
	from the embedding $\W\into\L[2]$ along with \eqref{b-fracexp}. The desired result then follows immediately by applying the bound in \eqref{b-c1-2} to the corresponding term in \eqref{b-c1-1}.

	{\bf Step 2:} We next prove $Y'(t)\geq  C_2\left[G(t)+ ||u'(t)||_2^2 + ||u(t)||_{1,p}^p \right]$ a.e. $[0,T]$ for a constant $C_2>0$.\\
	Using the expression of $N''$ from \eqref{blowup-npp2},
	\begin{align}\label{b-c2-1}
	Y'(t)&=(1-\alpha)G(t)^{-\alpha}G'(t) + \beta N''(t)\notag \\
	 &=(1-\alpha)G(t)^{-\alpha}G'(t) + 2\beta||u'(t)||_2^2 - 2\beta||u(t)||_{1,p}^p 
	 - 2\beta \langle -\Delta_2 u'(t),u(t)\rangle_2\notag \\
	  &\qquad\qquad+2\beta(q+1)||u(t)||_{q+1}^{q+1} + 2\beta(r+1)|\t u(t)|_{r+1}^{r+1} 
	\end{align}
	Taking $m=\min\{q,r\}$, 
	\begin{align*}
	&2\beta(q+1)||u(t)||_{q+1}^{q+1} + 2\beta(r+1)|\t u(t)|_{r+1}^{r+1}\notag \\
	&\qquad\qquad\geq 2\beta(m+1)S(t)\notag \\
	&\qquad\qquad\geq 2\beta(m+1)G(t) + \underbrace{\beta(m+1)||u'(t)||_2^2 + \frac{2\beta(m+1)}{p}||u(t)||_{1,p}^p}_{2\beta(m+1)\E(t)}
	\end{align*}  
    from the energy inequality \eqref{b-energy2}. Applying this estimate to \eqref{b-c2-1}, 
    \begin{multline}\label{b-c2-2}
    Y'(t) \geq (1-\alpha)G(t)^{-\alpha}G'(t) + \beta(m+3)||u'(t)||_2^2 \\
    + 2\beta\left(\frac{m+1}{p}-1 \right)||u(t)||_{1,p}^p 
    +2\beta(m+1)G(t) 
    - 2\beta\langle  -\Delta_2 u'(t),u(t)\rangle_2.    
    \end{multline} 
    Since the minimum of $q+1$ and $r+1$ is still greater than $p$ by assumption, the coefficient of $||u(t)||_{1,p}^p$ in this expression is indeed positive.  To bound the remaining negative term in \eqref{b-c2-2} we find that from the operator norm bound in (\ref{2.4}) and Young's inequality with $\epsilon G(t)^\alpha$ that  
	\begin{align}\label{b-c2-3}
	2\beta \langle -\Delta_2 u'(t),u(t)\rangle_2 &\leq 4\beta ||u'(t)||_{1,2}||u(t)||_{1,2}\notag \\
	&\leq \frac{4\beta}{2\epsilon G(t)^\alpha}||u'(t)||_{1,2}^2 + \frac{4\beta\epsilon G(t)^\alpha}{2}||u(t)||_{1,2}^2\notag \\
	&=2\beta\epsilon G(t)^{-\alpha}G'(t) + 2\beta\epsilon G(t)^\alpha||u(t)||_{1,2}^2. 
	\end{align}
	The latter of these summands may be further controlled by the energy inequality \eqref{b-energy2}, since 
	\begin{align*}%
		2\beta\epsilon G(t)^\alpha||u(t)||_{1,2}^2
		&\leq 2\beta\epsilon S(t)^\alpha||u(t)||_{1,2}^2\notag \\
		&= 2\beta\epsilon\left( ||u(t)||_{q+1}^{\alpha(q+1)} + |\t u(t)|_{r+1}^{\alpha(r+1)}\right)||u(t)||_{1,2}^2\notag \\
		&\leq2C\beta\epsilon\left( ||u(t)||_{1,p}^{\alpha(q+1)+2} + ||u(t)||_{1,p}^{\alpha(r+1)+2} \right) 
	\end{align*}
		from the embeddings $\W[1,p]\into\L[q+1]$ and the trace $\W[1,p]\tinto\Lb[r+1]$ mentioned in Remark~\ref{rmk:embeddings} along with the embedding $\W[1,p]\into\W[1,2]$.  By choosing $\alpha$ sufficiently small (say, $\alpha<\min\{(p-2)/(q+1),\, (p-2)/(r+1)\}$) %
		we may apply the bound in \eqref{b-fracexp} so that 
		\begin{align}\label{b-c2-5}
		2\beta\epsilon G(t)^\alpha||u(t)||_{1,2}^2 &\leq 2C\beta\epsilon\left[(||u(t)||_{1,p}^p)^\frac{\alpha(q+1)+2}{p}
		+ (||u(t)||_{1,p}^p)^\frac{\alpha(r+1)+2}{p}  \right]\notag \\
		&\leq 2C\beta\epsilon(G(t)+||u(t)||_{1,p}^p).
		\end{align}
		By applying \eqref{b-c2-5} to \eqref{b-c2-3} and in turn using this bound in \eqref{b-c2-2} we obtain 
		\begin{multline*}
		Y'(t)\geq (1-\alpha-2\beta\epsilon)G(t)^{-\alpha}G'(t) + \beta(m+3)||u'(t)||_2^2 \\
		+2\beta\left( \frac{m+1}{p}-1- C\epsilon \right)||u(t)||_{1,p}^p + 2\beta(m+1 -C\epsilon)G(t).
		\end{multline*}
		Since no further adjustment of $\alpha$ is necessary we may select $\epsilon>0$ so that 
		\begin{gather*}
			C\epsilon<\frac{m+1}{p}-1\quad\text{and}\quad C\epsilon<m+1
		\end{gather*}
		and then take $\beta$ sufficiently small so that $1-\alpha-2\beta\epsilon>0$.  This yields the desired result, since 
		\begin{multline*}
		Y'(t)\geq \beta(m+3)||u'(t)||_2^2 
		+2\beta\left( \frac{m+1}{p}-1- C\epsilon \right)||u(t)||_{1,p}^p + 2\beta(m+1 -C\epsilon)G(t)	
		\end{multline*}
		with positive coefficients on every term.

	Finally, since $G(0)>0$ and $$Y(0)=G(0)^{1-\alpha} + \beta N'(0)$$
	we may always, if necessary, select a smaller positive value of $\beta$ so that $Y(0)>0$. Thus, combining the results of Steps 1 and 2 we have the desired result  \eqref{b-wts}.  That is, $Y$ satisfies the ordinary differential inequality 
	\begin{align*}
		\begin{cases} Y'(t)\geq CY(t)^{1/(1-\alpha)}\\
		Y(0)>0
		\end{cases}
	\end{align*}
	a.e. $[0,T)$ for a constant $C>0$, from which it follows from standard ODE theory that the maximal interval of existence of $Y$ is the finite interval $[0,T)$ with 
	$$T<\frac{Y(0)^{(1-\alpha)/\alpha}}{C},$$
	and that 
	\begin{align*}
		\infty=\limsup_{t\to T^-}Y(t)
		=\limsup_{t\to T^-}\left( G(t)^{1-\alpha} + N'(t) \right).
		\end{align*}
	At least one of the following conditions is therefore met:
	\begin{subequations}
		\begin{alignat}{3}
		\infty &= \limsup_{t\to T^-}G(t)&&=\limsup_{t\to T^-}\int_0^t ||u'(\tau)||_{1,2}^2\,d\tau,\label{blow-casea}\\
		\infty&=\limsup_{t\to T^-}N'(t)&&\leq \limsup_{t\to T^-}  ||u(t)||_2||u'(t)||_2.\label{blow-caseb}
		\end{alignat}\end{subequations}
		
	In either case it is clear that the lifespan of $u$ is at most $T$ in order to accord with items \eqref{def-a} and \eqref{def-b} of Definition~\ref{def:weaksln}.  Finally, assume for contradiction that $\limsup_{t\to T^-}\E(t)<\infty$.  The energy inequality \eqref{b-energyid} along with the embedding $\W\into\L[q+1]$ and trace $\W\tinto\Lb[r+1]$ in line with Remark~\ref{rmk:embeddings} then implies 
	\begin{align*}
	\int_0^t ||u'(\tau)||_{1,2}^2\,d\tau &\leq E(0) - E(t)\\
	&=\E(0)- ||u_0||_{q+1}^{q+1} - ||\t u_0||_{r+1}^{r+1}  - \E(t) + ||u(t)||_{q+1}^{q+1} + |\t u(t)|_{r+1}^{r+1}\\
	&\leq C(1+||u(t)||_{q+1}^{q+1} + |\t u(t)|_{r+1}^{r+1})\\
	&\leq C(1+||u(t)||_{1,p}^{q+1} + ||u(t)||_{1,p}^{r+1})\\
	&\leq C(1+\E(t)^{(q+1)/p} + \E(t)^{(r+1)/p} )<\infty\\
	\end{align*} which precludes the condition in \eqref{blow-casea} from occurring.  Simultaneously, from the embedding $\W\into\L[2]$ along with Young's inequality we see that  
	$$||u(t)||_2||u'(t)||_2 \leq C(||u(t)||_{1,p}^p + 1) + ||u'(t)||_2^2 \leq C(\E(t)+1)<\infty$$
	which precludes the condition in \eqref{blow-caseb} from occurring.	As at least one of \eqref{blow-casea} and \eqref{blow-caseb} must hold, it is therefore the case that $\limsup_{t\to T^-}\E(t)=\infty$ which establishes the desired result.
\end{proof}

\bibliographystyle{abbrv}
\bibliography{mohnick}
\end{document}